\documentclass[11pt,a4paper]{article}
\usepackage[style=numeric,backend=biber]{biblatex} 
\bibliography{references.bib}

\usepackage{hyperref}
\usepackage{titlesec}
\titleformat{\paragraph}
{\normalfont\normalsize\bfseries}{\theparagraph}{1em}{}
\titlespacing*{\paragraph}
{0pt}{3.25ex plus 1ex minus .2ex}{1.5ex plus .2ex}
\usepackage{mathrsfs}
\usepackage[dvips]{epsfig}
\usepackage{graphicx}
\usepackage{amssymb, graphics, amsmath, amsthm, amsopn, amstext, amsfonts, algorithm, algorithmic}
\usepackage{subcaption}
\usepackage{fancyhdr}
\usepackage{lipsum} 
\usepackage{booktabs} 
\usepackage{tabularx}
\usepackage{caption}
\usepackage{adjustbox} 
\usepackage{rotating}
\usepackage{color}
\usepackage[left=3cm,right=2cm,top=3cm,bottom=2cm]{geometry}

\usepackage{multirow}
\usepackage{multicol}
\usepackage{cases}
\usepackage[utf8]{inputenc}
\usepackage[T1]{fontenc}
\usepackage{authblk}

\newcommand{\wt}{\widetilde}

\newcommand{\supp}{ \mbox{supp}}

\newcommand{\Argmin}{\mbox{Argmin}}

\newcommand{\Ical}{\mathcal I}
\newcommand{\Jcal}{\mathcal J}
\newcommand{\Lcal}{\mathcal L}

\newcommand{\sgn}{{\rm sgn}}

\newcommand{\mycut}[1]{{}}

\newtheorem{theorem}{Theorem}[section] 
\newtheorem{lemma}{Lemma}[section] 
\newtheorem{definition}{Definition}[section]

\makeatletter
\def\thanks#1{\protected@xdef\@thanks{\@thanks
        \protect\footnotetext{#1}}}
\makeatother

\begin{document}
\title{\textbf{
Sparse Extended Mean-Variance-CVaR Portfolios with Short-selling}}

\author[1]{Ahmad Mousavi$^*$\thanks{Emails: \{mousavi and boukouva\}@american.edu, salahim@guilan.ac.ir,
}} 
\affil{Department of Mathematics and Statistics, American University}
\author[2]{Maziar Salahi\thanks{}}
\affil[2]{Department of Applied Mathematics, Faculty of Mathematical Sciences, University of Guilan}
\author[1]{Zois Boukouvalas\thanks{}}

\date{\today}
\date{}
\maketitle
\begin{abstract}
This paper introduces a novel penalty decomposition algorithm customized for addressing the non-differentiable and nonconvex problem of extended mean-variance-CVaR portfolio optimization with short-selling and cardinality constraints. The proposed algorithm solves a sequence of penalty subproblems using a block coordinate descent (BCD) method while striving to fully exploit each component within the objective function and constraints. Through rigorous analysis, the well-posedness of each subproblem of the BCD method is established, and closed-form solutions are derived where possible. A comprehensive theoretical convergence analysis is provided to confirm the efficacy of the introduced algorithm in reaching a Lu--Zhang minimizer for this intractable optimization problem. Numerical experiments conducted on real-world datasets validate the practical applicability and effectiveness of the introduced algorithm based on various criteria. Notably, the existence of closed-form solutions within the BCD subproblems prominently underscores the efficiency of our algorithm when compared to state-of-the-art methods.
\end{abstract}

\section{Introduction}
The mean-variance CVaR problem is crucial in finance, as it provides a robust framework for optimizing portfolios by considering both expected returns (mean) and associated risks (variance and conditional value at risk). This approach helps investors balance the trade-off between maximizing returns and minimizing losses, enhancing portfolio performance and risk management in volatile markets
\cite{dixit2020project,lwin2017mean, van2021surprising}. Including short-selling is essential, as it allows investors to profit from downward market movements, improving portfolio diversification and risk management by providing additional opportunities for profit generation in declining markets \cite{wang2023distributionally}.

{\color{black}Moreover, sparse optimization plays a critical role in various fields, such as machine learning, signal processing, and finance, offering efficient solutions for problems where the underlying data or parameters exhibit sparsity \cite{cesarone2013new,du2022high,steuer2024computing, xia2023high, wu2024sparse}.  In portfolio optimization, having only a few assets can significantly impact diversification, potentially increasing exposure to individual asset risk and volatility. This highlights the importance of careful asset selection and risk management strategies. Hence, the cardinality constraints have been imposed in various portfolio optimization models  \cite{deng2024unified,larni2024bi}.
}

In this paper, we tackle the following extended mean-variance-CVaR portfolio problem, which includes short-selling and nonconvex cardinality constraints:
\begin{equation}\label{pr: p-original} \tag{\mbox{$P$}}
\begin{aligned}
\min_{x\in \mathbb{R}^n,\gamma\in \mathbb R}\quad & \ 
f(x,\gamma):=\lambda_1 x^TAx - \lambda_2(\Psi^Tx-\delta\|x-\phi\|_1) 
+\lambda_3(\gamma+\frac{1}{m(1-\beta)}
\sum_{j=1}^m (-d_j^Tx-\gamma)^+)
\\ \textrm{subject to} \quad & 
 \textbf{e}^Tx=1, \quad 
L\le x\le U,    \quad
\left\{
\begin{array}{ll}
x_i\ge 0; \,  \forall i\in \Gamma^+:= \{ i \in [n] | \, \eta_i>0\},\\
x_i\le 0; \,  \forall i\in \Gamma^-:= \{ i \in [n] | \, \eta_i<0\},\\
x_i\in \mathbb R; \,  \forall i\in \Gamma^0:= \{ i \in [n] | \, \eta_i=0\},\\
\end{array}\right. \quad \text{and} \quad   
\|x\|_0\le k,   
\end{aligned}
\end{equation}
{\color{black}where $\lambda_1$ and $\lambda_2\ge 0$, and $\lambda_3=1-\lambda_1-\lambda_2\ge 0$ are the weights of variance, return, and CVaR, respectively.   $\beta\in (0,1)$ is the confidence level, $\delta \ge 0$ is the cost per change in the proportion of the $i$-th asset, $L, U \in \mathbb R^n$ and $L< U$ are the lower and upper bounds for the weight vector of assets $x$, and  $A\succeq 0$ is the covariance matrix.
Also we have
$\Psi= \mu+r_ch, \eta = \mu$, and 
$d_j:= b_j-r_ch; \ \forall j\in [m]$,
where $\mu$ is the return vector, $r_c\ge 0$ is the risk-neutral interest, $h_j$ is the portion of risk free return for the investor when stock $i$ is on a short selling position. The sign constraint on $x$ is essentially $x_i\eta_i\ge 0; \forall i\in [n]$, which means for each stock that is not in a short-selling position, the proportion of investment is positive. Finally,  $b_j,\ j \in [m] $ are the generated scenarios and $\phi_i$ is the proportion of the initial wealth invested in the $i$th stock.

Problem (\ref{pr: p-original}) has been the subject of several research \cite{hamdi2024penalty,csen2024sparsity, shi2019optimal}. When $\lambda_1=0$, it reduces to the widely studied mean-CVaR with a cardinality constraint, see, for example, \cite{ferreira2021mean,kobayashi2021bilevel} and references therein.  It is worth noting that by minimizing both Variance and CVaR simultaneously, we are essentially seeking a portfolio that exhibits low overall volatility (variance) and resilience to crises (CVaR). In the event of a major market shock (such as the 2020 COVID crash), the portfolio is explicitly designed to limit the average severity of losses. This is why we use two measures of risk together.}

Recently, Hamdi et al. \cite{hamdi2024penalty} applied the penalty decomposition method (PDM) to solve (\ref{pr: p-original}) and reported promising numerical results. Nevertheless, the full structure of the problem was not exploited. 
Therefore, we introduce a novel penalty decomposition algorithm customized to fully leverage every component of the objective function in (\ref{pr: p-original}) along with its constraints. We solve a sequence of penalty subproblems, of which a saddle point is efficiently identified using a BCD method. We meticulously examine each corresponding subproblem within this BCD method to derive closed-form solutions whenever feasible. Specifically, to handle the $\|.\|_1$ norm in the objective function along with bound constraints, we introduce a new variable. Additionally, another variable is required to manage the softmax function in the objective function. Lastly, we resort to a third variable to address the cardinality constraint, along with the short-selling constraints. Our discussions show how introducing these new variables leads to structured subproblems, which, in turn, admit closed-form or otherwise efficient solutions. 
We rigorously establish the convergence of our introduced algorithm toward a Lu--Zhang \cite{lu2013sparse} minimizer of the non-differentiable problem (\ref{pr: p-original}) with non-convex constraints.

Section \ref{sec: method} introduces Algorithm \ref{algo: PD_method_p}, a penalty decomposition algorithm designed to tackle the non-differentiable and nonconvex problem (\ref{pr: p-original}). This algorithm employs a block coordinate Algorithm \ref{algo: BCD} to address each penalty subproblem. We establish the well-posedness of these subproblems and derive closed-form solutions for three of them. The convergence analysis of both Algorithm \ref{algo: BCD} and Algorithm \ref{algo: PD_method_p} is established in Section \ref{sec: convergence}, demonstrating that our introduced algorithm effectively reaches a Lu--Zhang minimizer of (\ref{pr: p-original}). Additionally, Section \ref{sec: numerical} presents extensive numerical results obtained from real-world data. {\color{black}To improve the readability of the paper, we have moved the proofs of some lemmas and theorems to the appendix.}

\textbf{Notation}.
The complement of a set $S$ is denoted as $S^c$. We use $|S|$ to represent its cardinality. For a natural number $n$, we define $[n]$ as the set $\{1,2,\dots,n\}$.
Now, consider a set $S$ given by $S=\{i_1,i_2,\dots, i_{|S|}\}$, which is a subset of $[n]$. For any vector $x$ in $\mathbb R^n$, we denote the coordinate projection of $x$ with respect to the indices in $S$ as $[x_S;0]$, which means that the $i$th element of this vector equals $x_i$ when $i$ belongs to $S$, and it equals $0$ for $i$ in $S^c$.
We determine whether a matrix $A$ is positive semidefinite or definite by the notations $A \succeq 0$ and $A \succ 0$, respectively. In this paper, we let $\sgn(a):=1$ for $a>0, \sgn(a):=-1$ for $a<0,$ and $\sgn(0):=0$. For $x$ and $y\in \mathbb R^n$,  $x\circ y$ shows the Hadamard (element-wise) multiplication of $x$ and $y$. Recall that \cite{boyd2004convex}
$$\partial (\|\cdot\|_1)|_{ x} \, = \, J_1\times J_2 \times \cdots \times J_n, \quad \mbox{with} \quad J_k = \left\{
\begin{array}{ll}
[-1,1] & \mbox{if}\quad    x_i = 0,\\
\{\sgn(x)\} &  \mbox{if}\quad  x_i \ne 0,\\
\end{array}\right.
$$
and
$$\partial (\cdot)^+|_{ x} \, = \, J_1\times J_2 \times \cdots \times J_n, \quad \mbox{with} \quad J_k = \left\{
\begin{array}{ll}
[0,1] & \mbox{if}\quad  x_i = 0,\\
\{1\} & \mbox{if}\quad    x_i > 0,\\
\{0\} & \mbox{if}\quad   x_i < 0.\\
\end{array}\right.
$$

\section{An Efficient Customized Penalty Decomposition Algorithm}  \label{sec: method} 

Here, we propose our customized penalty decomposition algorithm for solving (\ref{pr: p-original}) that fully exploits all the available structures of the objective function and constraints of this problem. 

\subsection{Methodology} 
In this subsection, we elaborate on how we design our customized penalty decomposition algorithm. Observe that we can equivalently reformulate this nonconvex problem as follows:
\begin{equation} \label{pr: equivalent_p}
\begin{aligned}
\min_{x,y,z,w\in \mathbb{R}^n, \gamma\in \mathbb R}\quad & \ 
 \lambda_1 x^TAx - \lambda_2\Psi^Tx +\lambda_2\delta\|z-\phi\|_1 +
 \lambda_3 (\gamma+\frac{1}{m(1-\beta)}
\sum_{j=1}^m (-d_j^Tw-\gamma)^+)
\\ \textrm{subject to} \quad & 
 \textbf{e}^Tx=1, \quad 
y_i\ge 0; \,  \forall i\in \Gamma^+, \quad y_i\le 0; \,  \forall i\in \Gamma^-, \quad
\|y\|_0\le k,  \quad L\le z\le U,    \\ & x-y=0, \quad  x-z=0, \quad\text{and} \quad  x-w=0.
\end{aligned}
\end{equation}
Specifically, we introduce $y$ to deal with sparsity,  $z$ to handle $\|.\|_1$, and $w$ to effectively manage the last term of the objective function in (\ref{pr: p-original}). The constraints are also decoupled accordingly using these new variables, so that we can possibly obtain closed-form solutions or, if not, much simpler penalty subproblems in our customized penalty decomposition algorithm (see  Subsection \ref{subsec: detaile_subproblems}).  

More precisely, suppose that
 \begin{align} \label{def: q_rho}
q_{\rho}(x,y,z,w,\gamma):= &\,  \lambda_1 x^TAx - \lambda_2\Psi^Tx +\lambda_2\delta\|z-\phi\|_1\notag
+ \lambda_3 (\gamma+\frac{1}{m(1-\beta)}
\sum_{j=1}^m (-d_j^Tw-\gamma)^+)\notag  \\ & +\rho(\|x-y\|_2^2
+\|x-z\|_2^2+\|x-w\|_2^2),
\end{align}
and 
 \begin{align} 
\mathcal X:=& \, \{x\in \mathbb R^n \, | \, \textbf{e}^Tx=1\},\notag
\\ 
\mathcal Y:=& \,  \{y\in \mathbb R^n \, | \, y_i\ge 0; \,  \forall i\in \Gamma^+, \quad y_i\le 0; \,  \forall i\in \Gamma^-, \quad \text{and} \quad  \|y\|_0\le k\} \notag  \\ 
 \mathcal{Z}:=& \, \{z \in \mathbb R^n \, | \, L\le z\le U\} \notag \\
  \mathcal{W}:=& \, \{w \in \mathbb R^n \, | \, \textbf{e}^Tw=1 \quad \text{and} \quad  L\le w\le U\}. \notag
\end{align}
In our method, we consider a sequence of penalty subproblems as follows:
\begin{equation} \label{pr: p-x,y,z,w,gam} \tag{\mbox{$P_{x,y,z,w,\gamma}$}}
\min_{x,y,z,w,\gamma} \quad q_{\rho}(x,y,z,w,\gamma)  \qquad
\textrm{subject to} \qquad 
 x\in \mathcal{X}, \quad  y\in \mathcal{Y}, \quad z\in \mathcal{Z}, \quad w\in \mathcal{W},  \quad \text{and}  \quad \gamma\in \mathbb R
\end{equation}
The idea is that by gradually increasing the value of $\rho$ towards infinity, we can effectively address the optimization problem (\ref{pr: equivalent_p}). It is important to emphasize that (\ref{pr: p-x,y,z,w,gam}) is yet nonconvex. However, the following BCD method efficiently converges to a saddle point of it (see Theorem \ref{thm: BCD_convergence}).

\begin{algorithm}[H]
\caption{BCD Method for Solving (\ref{pr: p-x,y,z,w,gam})}
\begin{algorithmic}[1]
\label{algo: BCD}
\STATE Input: Select arbitrary $y_{0}\in \mathcal Y, z_{0}\in \mathcal Z,$ and $ (w_0,\gamma_0)\in \mathcal{W} \times \mathbb R$.
\STATE Set $l=0$.
\STATE 
$x_{l+1}=\Argmin_{x\in \mathcal X} \ q_{\rho}(x,y_l,z_l,w_l, \gamma_l)$. 
\STATE 
$y_{l+1}\in \Argmin_{y\in \mathcal Y} \ q_{\rho}(x_{l+1},y,z_l,w_l, \gamma_l).$
\STATE 
$z_{l+1}=\Argmin_{z\in \mathcal Z} \ q_{\rho}(x_{l+1},y_{l+1},z, w_l, \gamma_l).$
\STATE 
$(w_{l+1},\gamma_{l+1})\in \Argmin_{(w,\gamma)\in \mathcal{W}\times \mathbb R} \ q_{\rho}(x_{l+1},y_{l+1},z_{l+1}, w,\gamma).$
\STATE $l \leftarrow l+1$ and go to step (3).
\end{algorithmic}
\end{algorithm}

We are now prepared to introduce our penalty decomposition algorithm, which begins with a positive penalty parameter and gradually increases it until convergence is achieved. Algorithm \ref{algo: BCD} handles the corresponding subproblem for a fixed $\rho$. For the problem (\ref{pr: p-original}), we assume that we have a feasible point denoted as $x^{\text{feas}}$ in hand, which is easy to obtain.
To present this algorithm and its subsequent analysis, we define:
\[
\begin{aligned}
\Upsilon \;&\ge\;
\max\!\Big\{ f(x^{\text{feas}},\gamma^{\text{feas}}),\;
              \min_{x\in \mathcal{X}}
              q_{\rho^{(0)}}\!\big(x, y^{(0)}_0, z^{(0)}_0, w^{(0)}_0, \gamma^{(0)}_0\big)
     \Big\} \;>\; 0,\\[7pt]
X_\Upsilon \;&:=\;
\Big\{ (x,\gamma)\in \mathbb{R}^n\times\mathbb{R}\ \Big|\ f(x,\gamma)\le \Upsilon \Big\}.
\end{aligned}
\]

\begin{algorithm}[H]
\caption{A Customized Penalty Decomposition Algorithm for Solving (\ref{pr: p-original})}
\begin{algorithmic}[1]
\label{algo: PD_method_p}
\STATE Inputs:  $r>1, \rho^{(0)}>0$, and  $y^{(0)}_0\in \mathcal Y, z^{(0)}_0\in \mathcal Z,  w^{(0)}_0 \in \mathcal{W},$ and $ \gamma^{(0)}_0\in \mathbb R$
\STATE Set $j=0$.
\REPEAT
 \STATE Set $l=0$.
 \REPEAT
\STATE Solve $x^{(j)}_{l+1}=\Argmin_{x\in \mathcal X} \ q_{\rho^{(j)}}(x,y^{(j)}_l,z^{(j)}_{l}, w^{(j)}_{l}, \gamma^{(j)}_{l})$ [see (\ref{eqn: x_*})].
\STATE Solve
$y^{(j)}_{l+1}\in \Argmin_{y\in \mathcal Y} \ q_{\rho^{(j)}}(x^{(j)}_{l+1},y,z^{(j)}_{l}, w^{(j)}_{l}, \gamma^{(j)}_{l})$ [see (\ref{eqn: y_*})].
 \STATE Solve  $z^{(j)}_{l+1}=\Argmin_{z\in\mathcal Z}\ q_{\rho^{(j)}}(x^{(j)}_{l+1},y^{(j)}_{l+1},z, w^{(j)}_{l},\gamma^{(j)}_{l})$ [see (\ref{eqn: z_*})].
  \STATE Solve  $(w^{(j)}_{l+1},\gamma^{(j)}_{l+1})\in \Argmin_{(w,\gamma)\in \mathcal{W} \times \mathbb R}\ q_{\rho^{(j)}}(x^{(j)}_{l+1},y^{(j)}_{l+1},z^{(j)}_{l+1},w,\gamma)$ [see (\ref{pr: pw-equivalent})].
 \STATE Set $l \leftarrow l+1$.
 \UNTIL{stopping criterion (\ref{BCD-practical-stopping-criterion}) is met.} 

\STATE  Set $\rho^{(j+1)} = r\cdot \rho^{(j)}$.

 \STATE Set
 $(x^{(j)},y^{(j)} ,z^{(j)},w^{(j)}, \gamma^{(j)}):= (x^{(j)}_{l}, y^{(j)}_{l},z^{(j)}_{l}, w^{(j)}_{l}, \gamma^{(j)}_{l})$.

\STATE 
If $\min_{x\in \mathcal{X}} q_{\rho^{(j+1)}}(x,y^{(j)},z^{(j)}, w^{(j)},\gamma^{(j)})> \Upsilon$, then $y_0^{(j+1)}=z_0^{(j+1)}=w_0^{(j+1)}=x^{\text{feas}}$ and $\gamma^{(j+1)}_0=\gamma^{\text{feas}}$. Otherwise,
$y^{(j+1)}_0 = y^{(j)}, z^{(j+1)}_0 = z^{(j)}, w^{(j+1)}_0 = w^{(j)}$, and $\gamma^{(j+1)}_0 = \gamma^{(j)}$.
\STATE Set $j \leftarrow j+1$.
\UNTIL{stopping criterion (\ref{outer loop stopping criteria}) is met}. 
\end{algorithmic}
\end{algorithm}
We stop the inner loop if 
\begin{equation} \label{BCD-practical-stopping-criterion}
\max
\left\{
\frac{\|x_l-x_{l-1}\|_\infty}{\max \left(\|x_l\|_\infty,1 \right)},
\frac{\|y_l-y_{l-1}\|_\infty}{\max \left(\|y_l\|_\infty,1 \right)},
\frac{\|z_l-z_{l-1}\|_\infty}{\max \left(\|z_l\|_\infty,1 \right)},
\frac{\|w_l-w_{l-1}\|_\infty}{\max \left(\|w_l\|_\infty,1 \right)},
\frac{|\gamma_l-\gamma_{l-1}|}{\max \left(|\gamma_l|,1 \right)}
\right\}
\le
\epsilon_I,
\end{equation}
and the outer loop when a convergence criterion is met:
\begin{equation} \label{outer loop stopping criteria}
\|x^{(j)}-y^{(j)}\|_{\infty}+\|x^{(j)}-z^{(j)}\|_{\infty}
+\|x^{(j)}-w^{(j)}\|_{\infty}
\le
\epsilon_O.
\end{equation}

\subsection{Subproblems of Algorithm \ref{algo: BCD}} \label{subsec: detaile_subproblems}
We discuss how to efficiently solve the constrained subproblems presented in Algorithm \ref{algo: BCD} below.
\subsubsection{Subproblem of $x$}
This subproblem ($\min_{x\in\mathcal X}\ q_{\rho}(x,y,z,w,\gamma)$) becomes the following convex quadratic optimization problem:
\begin{equation} \label{pr: px} \tag{\mbox{$P_x$}}
\begin{aligned}
\min_{x\in \mathbb{R}^n} \quad & 
\lambda_1x^TAx-\lambda_2\Psi^Tx+\rho (\|x-y\|_2^2+\|x-z\|_2^2+\|x-w\|_2^2)
 \qquad \textrm{subject to} \qquad
 \textbf{e}^Tx=1.
\end{aligned}
\end{equation}
\begin{lemma} \label{lem: solution_p_x}
The solution to the problem (\ref{pr: px}) is 
\begin{equation} \label{eqn: x_*} 
x_* = \frac{1}{2}B^{-1}\left( 
\vartheta + \frac{1-0.5\textbf{e}^TB^{-1}\vartheta}{0.5\textbf{e}^TB^{-1}\textbf{e}}\textbf{e}
\right),
\end{equation}
where 
\begin{equation} \label{eqns: B and vartheta} 
B:=\lambda_1A+3\rho I
\qquad \text{and} \qquad \vartheta:=2\rho (y+z+w)+\lambda_2\Psi.
\end{equation}
\end{lemma}
\subsubsection{Subproblem of $y$}
Recall that we let $\sgn(0)=0$, so by slightly abusing the notation, this subproblem ($ \min_{y\in\mathcal Y}\ q_{\rho}(x,y,z,w,\gamma)$) is as follows:
\begin{equation} 
\label{pr: py} \tag{\mbox{$P_y$}}
\min_{y\in \mathbb{R}^n} \ \|x-y\|_2^2 \qquad \mbox{subject to}  \qquad  \sgn(\eta)\circ y\ge 0, \quad \textrm{and} \quad \|y\|_0\le k.
\end{equation}
To provide a closed-form solution to the latter problem,  we first define the following generalized sparsifying operator.
\begin{definition} \label{def: gen_sparse}
Let $x$ and $\eta\in \mathbb R^n$ and a natural number $k\in [n]$  be given. Denote $ x^+=\max(x,0), \, x^-=\max(-x,0),\, \Gamma^+=\{i\, | \, \eta_i>0\},\, \Gamma^-=\{i\, | \, \eta_i<0\},$ and $\Gamma^0 =[n]-(\Gamma^+\cup\Gamma^{-})=\{i\, | \, \eta_i=0\}$. Then, the generalized sparsifying operator $\mathcal{S}_{k,\eta}(x)$ is defined as follows:
\begin{equation} \label{eqn: y_*} 
\mathcal{S}_{k,\eta}(x)= (
[x^+_{\Gamma^+};0]-[x^-_{\Gamma^-};0]+[x_{\Gamma^0};0])_\Jcal
\end{equation}
where $\Jcal$ is an index set corresponding to the $k$ largest components of 
$[x^+_{\Gamma^+};0]-[x^-_{\Gamma^-};0]+[x_{\Gamma^0};0]$
in absolute value.
\end{definition}

\begin{lemma} \label{lem: solution_p_y}
The solution to the problem (\ref{pr: py}) is  $y_*=\mathcal{S}_{k,\eta}(x)$ defined in (\ref{eqn: y_*}).
\end{lemma}

\subsubsection{Subproblem of $z$}

This subproblem ($\min_{z\in\mathcal Z}\ q_{\rho}(x,y,z,w,\gamma)$) becomes the following generalized soft thresholding operator problem:
\begin{equation} \label{pr: pz} \tag{\mbox{$P_z$}}
\begin{aligned}
\min_{z\in \mathbb{R}^n} \quad & 
{\rho}\|z-x\|_2^2+\lambda_2\delta\|z-\phi\|_1
 \qquad \textrm{subject to} \qquad
 L\le z\le U.
\end{aligned}
\end{equation}
\begin{lemma} \label{lem: solution_p_z} 
For $i\in [n]$, the solution to the problem (\ref{pr: pz}) is as follows: 
\begin{equation} \label{eqn: z_*}
z_{*i}:= \left\{ 
\begin{array}{ll}
x_i-\kappa  & \qquad \mbox{if \quad  $L_i\le  x_i-\kappa\le  U_i$} \qquad \text{and} \qquad
\mbox{$ x_i-\phi_i-\kappa>0$};
\\
 x_i+\kappa & \qquad \mbox{if \quad  $ L_i\le x_i+\kappa \le  U_i$}\qquad \text{and} \qquad
\mbox{$ x_i-\phi_i+\kappa<0$};
\\
0 &  \qquad \mbox{if \quad $|x_i-\phi_i|\le \kappa$} \qquad \text{and} \qquad
\mbox{$\phi_i\in [L_i,U_i]$};
\\
L_i &  \qquad \mbox{if} \quad
\left\{ \begin{array}{ll}
 \phi_i<x_i-\kappa < L_i, \qquad \mbox{or} &\\
 x_i+\kappa < \min(\phi_i, L_i), \qquad \mbox{or}& \\
|x_i-\phi_i|\le \kappa \qquad \mbox{and} \qquad
L_i>\phi_i;\end{array} \right.
\\
U_i &  \qquad \mbox{if} \quad
\left\{ \begin{array}{ll}
U_i< x_i+\kappa < \phi_i, \qquad \mbox{or}& \\
 x_i-\kappa > \max(U_i, \phi_i) \qquad \mbox{or} &\\
|x_i-\phi_i|\le \kappa \qquad \mbox{and} \qquad
\phi_i>U_i,\end{array} \right.
\end{array} \right. \end{equation}
where $\kappa = \lambda_2\delta/\rho$.
\end{lemma}
Even though  we obtained the closed-form solution of (\ref{pr: pz}) by analysis, we mention that its KKT conditions are:
\begin{equation*} \label{eqn: pz_KKT}
\begin{aligned}
& 2\rho (z_*-x)+\lambda_2\delta \partial\left( \|\cdot\|_1)\right|_{z_*-\phi}-T_1+T_2=0, \quad  0\le T_1\perp (z_*-L)\ge 0, \quad  \mbox{and}\quad  0\le T_2\perp (U-z_*)\ge 0. 
\end{aligned}
\end{equation*}

\subsubsection{Subproblem of $w$ and $\gamma$}
This subproblem ($\min q_\rho(x,y,z,w,\gamma))$ becomes the following:
\begin{equation}  \tag{$P_{(w,\gamma)}$}\label{pr: pw}
\begin{aligned}
\min_{(w, \gamma)\in \mathbb R^n \times \mathbb R}\quad & \ c(w,\gamma):=
 \lambda_3 (
\gamma+\frac{1}{m(1-\beta)}
\sum_{j=1}^m (-d_j^Tw-\gamma)^+)+\rho\|x-w\|_2^2
\\ \textrm{subject to} \quad &  \textbf{e}^Tw=1,    \quad  \text{and} \quad L\le w\le U.
\end{aligned}
\end{equation}
\begin{lemma} \label{lem: boundedness of p_w}
Let $\lambda_3 \in (0,1]$ (or equivalently $\lambda_1\ge 0, \lambda_2\ge 0$ with $\lambda_1+\lambda_2<1$) 
and $\beta\in (0,1)$. Then, the problem (\ref{pr: pw}) has a unique solution. 
\end{lemma}

The KKT conditions for $w_*$ are:
\begin{equation*}\label{pw_kkt}
\frac{- \lambda_3}{m(1-\beta)}
\sum_{j=1}^m d_j
\partial \left(.)^+\right|_{-d_j^Tw_*-\gamma}
+2\rho(w_*-x)=0, \quad \textbf{e}^Tw_*=1, \quad \mbox{and} \quad L\le w_*\le U.
\end{equation*}
Eventually, we point out that this convex problem can be solved using a standard solver because it can be transformed into the following convex quadratic program with linear constraints:
\begin{equation} \label{pr: pw-equivalent} 
\begin{aligned}
\min_{w\in \mathbb{R}^n,\gamma\in \mathbb R, t\in \mathbb R^m}\quad & \ 
 \lambda_3 (\gamma+\frac{1}{m(1-\beta)}
\sum_{j=1}^m t_j)+\rho\|x-w\|_2^2
\\ \textrm{subject to} \quad & 
 \textbf{e}^Tw=1, \quad 
L\le w\le U,    \quad 
t_j\ge 0, \quad \text{and} \quad  
t_j\ge -d_j^Tw-\gamma;\, \,  \forall j\in [m].
\end{aligned}
\end{equation}

\section{Convergence Analysis} \label{sec: convergence}
In this section, we begin by examining Algorithm \ref{algo: BCD} to address the penalty subproblem (\ref{pr: p-x,y,z,w,gam}) with a fixed $\rho\ge 1$. We demonstrate its efficacy in locating a saddle point for this subproblem. Subsequently, our attention shifts to Algorithm \ref{algo: PD_method_p}, where we illustrate its capability in identifying a convergent subsequence that converges to a Lu--Zhang minimizer of the original problem (\ref{pr: p-original}).

\subsection{Analysis of Algorithm \ref{algo: BCD}}
We analyze a sequence $\{(x_l,y_l,z_l,w_l,\gamma_l)\}$ generated by Algorithm \ref{algo: BCD} and provide a customized proof that any such sequence obtains a saddle point of (\ref{pr: p-x,y,z,w,gam}). This justifies the use of Algorithm \ref{algo: BCD} for this nonconvex problem.
\begin{lemma} \label{lem: BCD_boundedness}
Let  $\lambda_1$ and $\lambda_2 \ge 0,$ with $\lambda_3=1-\lambda_1-\lambda_2\ge 0, \beta\in (0,1), \delta\ge 0$, $\eta, L, U \in \mathbb R^n$ with $L< U$ element-wise, $d_j\in \mathbb R^n$ for each $j\in [m], A\succeq 0,$ and $\rho\ge 1$.  Consider the iterates of Algorithm \ref{algo: BCD}. Then, we have 
\begin{equation*}
    \max\{\|x_{l}\|,\|y_l\|, \|z_{l}\|,\|w_l\|, |\gamma_l|\}\le C,
\end{equation*}
where $C>0$ is a function of the parameters of (\ref{pr: p-original}), except for $\rho$.
\end{lemma}
This lemma proves that any sequence formed by Algorithm \ref{algo: BCD} is bounded; specifically, this bound is independent of $\rho$ whenever $\rho\ge 1$, which is the case for our penalty decomposition algorithm. Hence, every sequence produced by Algorithm \ref{algo: BCD} possesses at least one accumulation point.
The next theorem further confirms that each accumulation point is a saddle point of (\ref{pr: p-x,y,z,w,gam}).

{\color{black}
We note that a saddle point (block-coordinate minimizer) refers to a point 
$(x_*,y_*,z_*,w_*,\gamma_*)$ that minimizes $q_\rho$ with respect to 
each block separately while keeping the other blocks fixed (see~\eqref{bcd_saddle}). 
This notion is generally weaker than full first-order stationarity or the KKT 
conditions for the joint problem. If $q_\rho$ is continuously differentiable and each block-feasible set
$\mathcal X,\mathcal Y,\mathcal Z,\mathcal W$ is closed and convex,
then~\eqref{bcd_saddle} is equivalent to the first-order optimality
condition
\[
0 \in \nabla q_\rho(x_*,y_*,z_*,w_*,\gamma_*) \;+\;
N_{\mathcal X\times\mathcal Y\times\mathcal Z\times\mathcal W\times\mathbb R}
(x_*,y_*,z_*,w_*,\gamma_*),
\]
where $N_C(u)$ denotes the (convex) normal cone of a closed convex set
$C$ at $u$, so the limit point is (Clarke/KKT) stationary.
In our setting, however, $\mathcal Y$ is nonconvex due to the cardinality
constraint, so we do not claim full stationarity; Theorem~\ref{thm: BCD_convergence} asserts convergence to a saddle point (block-coordinate minimizer).

\begin{theorem} \label{thm: BCD_convergence}
Let $\{(x_l,y_l,z_l,w_l,\gamma_l)\}$ be a sequence generated by Algorithm \ref{algo: BCD} for solving (\ref{pr: p-x,y,z,w,gam}). First, $\left\{q_{\rho}(x_{l},y_{l},z_{l}, w_l, \gamma_l)\right\}$ is a non-increasing sequence. Second, any accumulation point $(x_*,y_*,z_*,w_*,\gamma_*)$ of $\{(x_l,y_l,z_l,w_l,\gamma_l)\}$ is a saddle point of the nonconvex problem (\ref{pr: p-x,y,z,w,gam}), that is, 
\begin{equation} \label{bcd_saddle}
\begin{aligned} 
x_* &\in \arg\min_{x\in\mathcal X} q_\rho(x,y_*,z_*,w_*,\gamma_*),\\
y_* &\in \arg\min_{y\in\mathcal Y} q_\rho(x_*,y,z_*,w_*,\gamma_*),\\
z_* &\in \arg\min_{z\in\mathcal Z} q_\rho(x_*,y_*,z,w_*,\gamma_*),\\
(w_*,\gamma_*) &\in \arg\min_{\substack{w\in\mathcal W\\ \gamma\in\mathbb R}} q_\rho(x_*,y_*,z_*,w,\gamma).
\end{aligned}
\end{equation}
\end{theorem}
\begin{proof}
By observing the definitions of $x_{l+1}, y_{l+1}, z_{l+1},$ and $(w_{l+1},\gamma_{l+1})$ in steps 6--9 of Algorithm \ref{algo: BCD}, we obtain
\begin{eqnarray} \label{eqn: q_xyz-inequality} 
  q_{\rho}(x_{l+1},y_{l+1},z_{l+1}, w_{l+1}, \gamma_{l+1}) & \le & q_{\rho}(x_{l+1},y_{l+1},z_{l+1}, w, \gamma); \qquad \forall w\in \mathcal W \ \mbox{and} \ \forall \gamma \in \mathbb R, \notag
  \\
    q_{\rho}(x_{l+1},y_{l+1},z_{l+1}, w_{l}, \gamma_{l}) & \le & q_{\rho}(x_{l+1},y_{l+1},z, w_{l}, \gamma_{l}); \qquad \forall z\in \mathcal Z, \notag
  \\ 
  q_{\rho}(x_{l+1},y_{l+1},z_l, w_{l}, \gamma_{l}) &\le &  q_{\rho}(x_{l+1},y,z_{l}, w_{l}, \gamma_{l}); \qquad \forall y\in \mathcal Y,\notag
  \\
  q_{\rho}(x_{l+1},y_{l},z_{l}, w_{l}, \gamma_{l})  &\le &  q_{\rho}(x,y_{l},z_l, w_{l}, \gamma_{l}); \qquad \forall x\in \mathcal X.
\end{eqnarray}
This leads to the following ($\forall l \in \mathbb N$):
\begin{equation*}
\begin{aligned}
q_{\rho}(x_{l+1},y_{l+1},z_{l+1}, w_{l+1}, \gamma_{l+1}) 
   & \le  q_{\rho}(x_{l+1},y_{l+1},z_{l+1},w_{l}, \gamma_{l})\\
   & \le q_{\rho}(x_{l+1},y_{l+1},z_{l},w_{l}, \gamma_{l})\\
     & \le q_{\rho}(x_{l+1},y_{l},z_{l},w_{l}, \gamma_{l})\\
   & \le q_{\rho}(x_{l},y_{l},z_{l},w_{l}, \gamma_{l}),
\end{aligned}
\end{equation*}
which shows that $q_{\rho}(x_{l},y_{l},z_{l}, w_l, \gamma_l)$ is a non-increasing sequence. 
Furthermore, by Lemma~\ref{lem: BCD_boundedness}, the iterates
$\{(x_l,y_l,z_l,w_l,\gamma_l)\}_{l\in\mathbb{N}}$ remain in a compact set.
Since $q_\rho$ is continuous, the sequence of objective values
$\{q_\rho(x_l,y_l,z_l,w_l,\gamma_l)\}_{l\in\mathbb{N}}$ is bounded below. Together with the fact established earlier that
$\{q_\rho(x_l,y_l,z_l,w_l,\gamma_l)\}_{l\in\mathbb{N}}$ is non-increasing,
it follows that this sequence converges.

Next, let $(x_*,y_*,z_*,w_*,\gamma_*)$ be an accumulation point of $\{(x_l,y_l,z_l,w_l,\gamma_l)\}_{l\in\mathbb N}$; i.e., along some infinite index set $\bar{\mathcal L}\subset\mathbb N$ we have $(x_l,y_l,z_l,w_l,\gamma_l)\to(x_*,y_*,z_*,w_*,\gamma_*)$. Since $\mathcal X,\mathcal Y,\mathcal Z,$ and $\mathcal W$ are closed, the limit is feasible. Because $\bar{\mathcal L}$ is infinite, the shifted set $\bar{\mathcal L}-1:=\{\,l-1:\ l\in\bar{\mathcal L},\,l\ge1\,\}$ is also infinite. Replacing $l$ by $l-1$ in \eqref{eqn: q_xyz-inequality} and letting $l\to\infty$ along $\bar{\mathcal L}$, continuity of $q_\rho$ yields:
\begin{equation*} 
\begin{aligned}
\lim_{l\to \infty}q_{\rho}(x_{l+1},y_{l+1},z_{l+1}, w_l, \gamma_l) & = \lim_{l\to \infty} q_{\rho}(x_{l+1},y_{l+1},z_{l}, w_l, \gamma_l) \\
   & =  \lim_{l\to \infty} q_\rho(x_l,y_l,z_l,w_l, \gamma_l)\\
   & = \lim_{l\in \bar {\mathcal{L}}\to \infty} q_\rho(x_l,y_l,z_l,w_l, \gamma_l) \\
   & = q_\rho(x_*,y_*,z_*, w_*, \gamma_*).
\end{aligned}
\end{equation*}
Taking limits in \eqref{eqn: q_xyz-inequality} along $\bar {\mathcal{L}}$ and using continuity of $q_\rho$ gives
\begin{eqnarray*} 
  q_{\rho}(x_{*},y_{*},z_{*}, w_*, \gamma_*) & \le & q_{\rho}(x_*,y_{*},z_*, w,\gamma); \qquad \forall w\in \mathcal W \ \mbox{and} \ \forall \gamma \in \mathbb R,\notag
  \\ 
  q_{\rho}(x_{*},y_{*},z_{*}, w_*, \gamma_*) & \le & q_{\rho}(x_*,y_{*},z, w_*, \gamma_*); \qquad \forall z\in \mathcal Z,\notag
  \\ 
  q_{\rho}(x_{*},y_{*},z_{*}, w_*, \gamma_*) &\le &  q_{\rho}(x_{*},y,z_{*}, w_*, \gamma_*); \qquad \forall y\in \mathcal Y,\notag
  \\
  q_{\rho}(x_{*},y_{*},z_{*}, w_*, \gamma_*)  &\le &  q_{\rho}(x,y_{*},z_*, w_*, \gamma_*); \qquad \forall x\in \mathcal X.
\end{eqnarray*}
Equivalently, \eqref{bcd_saddle} holds. Thus, $(x_*,y_*,z_*,w_*,\gamma_*)$ is a saddle point (block-coordinate minimizer) of (\ref{pr: p-x,y,z,w,gam}).
\end{proof}
}

\subsection{Analysis of Algorithm \ref{algo: PD_method_p}}
Here, we establish that our proposed customized penalty decomposition Algorithm \ref{algo: PD_method_p} obtains a Lu--Zhang minimum of the original nondifferentiable nonconvex problem (\ref{pr: p-original}). 
{\color{black}
For our analysis, we need Robinson's constraint qualification. Specifically, consider the general problem
\begin{equation}\label{prob:general}
\min_{x \in \mathcal X}\, F(x)
\qquad \text{subject to}\qquad G(x)\le 0,\quad H(x)=0,\quad \text{and} \quad \|x\|_0 \le k,
\end{equation}
where $\mathcal X \subseteq \mathbb R^n$ is a closed convex set, and $G:\mathbb R^n\to\mathbb R^m$, $H:\mathbb R^n\to\mathbb R^p$ are continuously differentiable. 
Let $x^*$ be a feasible point. Choose an index set $\mathcal L\subseteq\{1,\dots,n\}$ with $|\mathcal L|=k$ such that $x^*_j=0$ for all $j\notin\mathcal L$. Define the active inequality set 
\[A(x^*) := \{\, i \in \{1,\dots,m\} : G_i(x^*) = 0 \,\}.\]
With $\overline{\mathcal L} = \{1,\dots,n\}\setminus\mathcal L$, Robinson's constraint qualification at $x^*$ is the surjectivity condition below:
\begin{equation*}
\left\{
\begin{bmatrix}
G'(x^*)d - v\\[4pt]
H'(x^*)d\\[4pt]
(I_{\overline{\mathcal L}})^{\!T} d
\end{bmatrix}
:\; d\in T_{\mathcal X}(x^*),\; v\in\mathbb R^m,\; v_i\le 0\ \forall i\in A(x^*)
\right\}
= \mathbb R^m \times \mathbb R^p \times \mathbb R^{\,n-|\mathcal L|}.
\end{equation*}
Here $G'(x^*)$ and $H'(x^*)$ denote the Jacobians of $G$ and $H$ at $x^*$, and $I_{\overline{\mathcal L}}$ is the coordinate selection matrix extracting the zero coordinates. Further, for a closed convex set $\mathcal X$ and $x^*\in \mathcal X$, the tangent cone is
\[
T_{\mathcal X}(x^*)
=\operatorname{cl}\Big\{\, d:\ \exists\, t_k \downarrow 0,\ d_k \to d \ \text{with}\ x^*+t_k d_k \in \mathcal X \,\Big\}.
\]
%
Because all constraints in (\ref{pr: p-original}) except the sparsity constraint are affine, the linearized operator has full row span. Hence, condition~\eqref{eqn: robinson} holds automatically at any feasible point~$x^*$. Moreover, since $\mathcal X=\mathbb R^n$, we have $T_{\mathcal X}(x^*)=\mathbb R^n$. Hence, Robinson's constraint qualification for a Lu--Zhang minimizer $(x^*, \gamma^*)$ of (\ref{pr: p-original}) requires the existence of an index set $\Lcal \subseteq \{1,\dots,n\}$ with $|\Lcal|=k$ and $x^*_{\Lcal^c}=0$, ensuring that the following condition is satisfied \cite{lu2013sparse}:
\begin{eqnarray}\label{eqn: robinson} 
 \left\{  \begin{bmatrix} 
        -d - v \\ 
        d-\bar v 
        \\
        -d-\hat v
        \\
        d-\tilde v
        \\ \textbf{e}^T d\\ d_{\Lcal^c} \end{bmatrix} \   \big{|} \
\left\{
\begin{array}{ll} 
d \in \mathbb R^{n}\\
 v \in \mathbb R^n \quad  \mbox{s.t.} \quad v_i\le 0;  \,    \forall i\in \{i\, |\, x^*_i =L_i\} \\ 
\bar v \in \mathbb R^n \quad \mbox{s.t.} \quad  \bar v_i\le 0;\,  \forall i\in \{i \, | \, x^*_i =U_i\} \\
\hat v=[\hat v_{\Gamma^+}; 0]
\in \mathbb R^{n} \quad \mbox{s.t.} \quad \hat v_i \le 0; \,\forall i\in \Gamma^+ \cap  \Lcal^c
\\
\tilde v=[\tilde v_{\Gamma^-}; 0]
\in \mathbb R^{n} \quad \mbox{s.t.} \quad \tilde v_i \le 0; \,\forall i\in \Gamma^- \cap  \Lcal^c
\end{array}\right.
       \right\}
      = \mathbb R^n  \times \mathbb R^n \times \mathbb R^{n_1} \times \mathbb R^{n_2} \times \mathbb R \times \mathbb R^{|\Lcal^c|},\notag\\
\end{eqnarray} 
with $n_1=|\Gamma^+ \cap  \Lcal^c|$ and $n_2={|\Gamma^- \cap  \Lcal^c|}$.

Note that although $x^*$ does not appear explicitly in~\eqref{eqn: robinson}, since all constraints except the cardinality constraint are affine, the condition is entirely defined at $x^*$. Specifically, (i) the Jacobians are evaluated at $x^*$, (ii) the active set depends on which inequalities are tight at $x^*$, and (iii) the support set $\mathcal L$ is chosen to match the nonzero elements of $x^*$. Hence, Robinson's constraint qualification is a \emph{local} condition at the reference point.

}

Under these Robinson's conditions, the KKT  conditions for a Lu--Zhang minimizer $(x^*, \gamma^*)$ of (\ref{pr: p-original}) are the existence of Lagrangian multipliers $(\bar t, \bar T_1,\bar T_2,\bar T_5,\bar T_6)$  with $\bar t\in \mathbb R, \bar T_i\in \mathbb R^n; \forall i\in \{1,2,5,6\},$ and $\bar \Omega \in \mathbb R^n$ with $\mathcal{ L}\subseteq [n]$ such that $|\mathcal{L}|=k$ and the following holds:
\begin{equation}  \label{eqn: actual_KKT_conditions for p} 
\left\{
\begin{array}{ll}  
2\lambda_1 Ax^*-\lambda_2\Psi +\bar t\textbf{e}
-\bar T_5+\bar T_6
+\bar \Omega +\lambda_2\delta  
 \partial\left( \|\cdot\|_1)\right|_{x^*-\phi}-\bar T_1+\bar T_2 -\vspace{0.2cm}\\ 
\qquad \qquad \qquad \quad \qquad \qquad \qquad \qquad \qquad \qquad \qquad \qquad \ \quad \quad \quad \frac{ \lambda_3}{m(1-\beta)}\sum_{j=1}^md_j\partial \left(.)^+\right|_{-d_j^Tx^*-\gamma^*}\ni 0,
 \vspace{.1cm}
\\ 
\lambda_3-\frac{1}{m(1-\beta)}\sum_{j=1}^m
\partial \left(.)^+\right|_{-d_j^Tx^*-\gamma^*}\ni 0,
\vspace{.2cm}\\
\textbf{e}^Tx^*=1, \quad 0\le \bar T_1\perp (x^*-L)\ge 0, \quad 0\le \bar T_2\perp (U-x^*)\ge 0, \quad x^*_{{\mathcal{L}}^c}=0, \quad  \bar \Omega_{\mathcal{L}}=0, 
\vspace{0.1cm} \\
0\le \bar T_5 \perp [(x^*)^+_{\Gamma^+};0] \ge 0, \quad 
0\le \bar T_6 \perp -[(x^*)^-_{\Gamma^-};0] \ge 0,
\end{array}\right.
\end{equation}

\begin{theorem} \label{thm: PPC convergence}
Suppose that  $\lambda_1$ and $\lambda_2 \ge 0,$ with $\lambda_3=1-\lambda_1-\lambda_2>0, \beta\in (0,1), \delta\ge 0$, $\eta, L, U \in \mathbb R^n$ with $L< U$ element-wise, $d_j\in \mathbb R^n$ for each $j\in [m], A\succeq 0,$ and $\rho\ge 1$.
Let  $\left\{\left(x^{(j)}, y^{(j)},z^{(j)},w^{(j)},\gamma^{(j)}\right)\right\}$  be a sequence generated by Algorithm \ref{algo: PD_method_p} for solving (\ref{pr: p-original}). Then, the following holds:
 \begin{itemize}
  \item [(i)] $\left\{\left(x^{(j)}, y^{(j)},z^{(j)},w^{(j)},\gamma^{(j)}\right)\right\}$ has a convergent subsequence whose accumulation point  $(x^*, y^*,z^*, w^*, \gamma^*)$ satisfies $x^*=y^*=z^*=w^*$. Further, there exists an index subset $\Lcal\subseteq [n]$ with $|\Lcal|=k$ such that $x^*_{\Lcal^c}=0$.
  \item [(ii)] Suppose that Robinson's constraint qualification condition given in (\ref{eqn: robinson}) holds at $(x^*,\gamma^*)$ with the index subset $\Lcal$ indicated above. Then, $(x^*,\gamma^*)$ is a Lu--Zhang minimizer of (\ref{pr: p-original}).
 \end{itemize}
\end{theorem}
\begin{proof}
  
Due to Lemma \ref{lem: BCD_boundedness}, the sequence $\{(x^{(j)}, y^{(j)}, z^{(j)}, w^{(j)},\gamma^{(j)})\}$ is bounded and therefore, has a convergent subsequence. For our purposes, without loss of generality, we suppose that the sequence itself is convergent. Let  $\left(x^*, y^*,z^*,w^*,\gamma^*\right)$ be its accumulation point. 
Under the given assumptions, in view of a similar technique used in Lemma \ref{lem: boundedness of p_w}, we can show that 
$$ \lambda_1 {x^{(j)}}^TAx^{(j)} - \lambda_2\Psi^Tx^{(j)} +\lambda_2\delta\|z^{(j)}-\phi\|_1+ \lambda_3 (\gamma+\frac{1}{m(1-\beta)}
\sum_{j=1}^m (-d_j^Tw^{(j)}-\gamma)^+) \ge \hat C>-\infty.
$$
Thus, using definition (\ref{def: q_rho}) and step 14 of Algorithm \ref{algo: PD_method_p} leads to
$$
\rho^{(j)}\left(\|x^{(j)}-y^{(j)}\|^2+\|x^{(j)}-z^{(j)}\|^2+\|x^{(j)}-w^{(j)}\|^2 \right) \le \Upsilon -\hat C,
$$
and thus,
$$
\max\{\|x^{(j)}-y^{(j)}\|,\|x^{(j)}-z^{(j)}\|,\|x^{(j)}-w^{(j)}\|\}\le \sqrt{{(\Upsilon-\hat C)}/{\rho^{(j)}}}.
$$
Hence, $\max\{\|x^{(j)}-y^{(j)}\|,\|x^{(j)}-z^{(j)}\|, \|x^{(j)}-w^{(j)}\|\}\to 0$ when $\rho^{(j)}\to \infty$; proving that $x^*=y^*=z^*=w^*$.

Let $\mathcal{I}^{(j)} \subseteq [n]$ be defined such that $|\mathcal{I}^{(j)}| = k$, and $(y_{(\mathcal{I}^{(j)})^c})_i = 0$ for every $j \in \mathbb{N}$ and $i \in (\mathcal{I}^{(j)})^c$. Given that $\{\mathcal{I}^{(j)}\}$ is a bounded sequence of indices, it possesses a convergent subsequence. This implies the existence of an index subset $\mathcal{L} \subseteq [n]$ with $|\mathcal{L}| = k$ and a subsequence $\{(x^{(j_\ell)}, y^{(j_\ell)}, z^{(j_\ell)},w^{(j_\ell)},\gamma^{(j_\ell)})\}$ from the aforementioned convergent subsequence, such that $\mathcal{I}^{(j_\ell)} = \mathcal{L}$ for all sufficiently large $j_\ell$'s.
Consequently, because $x^* = y^*$ and $y^*_{\mathcal{L}^c} = 0$, it follows that $x^*_{\mathcal{L}^c} = 0$.

(ii)
Recall that $\{(x^{(j_\ell)}, y^{(j_\ell)}, z^{(j_\ell)},w^{(j_\ell)},\gamma^{(j_\ell)})\}$ is a saddle point of (\ref{pr: p-x,y,z,w,gam})
due to Theorem \ref{thm: BCD_convergence} and consequently, 
the KKT conditions of the problems (\ref{pr: px}), (\ref{pr: py}), (\ref{pr: pz}), and (\ref{pr: pw}) yield:

\begin{equation} \label{sequence_kkt} 
\left\{
\begin{array}{ll}
2\lambda_1 Ax^{(j_\ell)} -\lambda_2\Psi+ 2\rho^{(j_\ell)}  \underbrace{(3x^{(j_\ell)}-y^{(j_\ell)}-z^{(j_\ell)}-w^{(j_\ell)})}_{=(x^{(j_\ell)}-y^{(j_\ell)})+(x^{(j_\ell)}-z^{(j_\ell)})+(x^{(j_\ell)}-w^{(j_\ell)})} +t^{(j_\ell)} \textbf{e} =0,
 \vspace{0.25cm} \\
y^{(j_\ell)}_\Lcal=([x^+_{\Gamma^+};0]-[x^-_{\Gamma^-};0]+[x_{\Gamma^0};0])^{(j_\ell)}_\Lcal,
 \vspace{0.25cm}
 \\
2\rho^{(j_\ell)} (z^{(j_\ell)}-x^{(j_\ell)})+\lambda_2\delta \partial\left( \|\cdot\|_1)\right|_{z^{(j_\ell)}-\phi}-T^{(j_\ell)}_1+T^{(j_\ell)}_2\ni 0, 
  \vspace{0.25cm}
\\
\frac{- \lambda_3}{m(1-\beta)}
\sum_{j=1}^m d_j
\partial \left(.)^+\right|_{-d_j^Tw^{(j_\ell)}-\gamma^{(j_\ell)}}
+2\rho(w^{(j_\ell)}-x^{(j_\ell)})-T^{(j_\ell)}_3+T^{(j_\ell)}_4\ni 0, 
\vspace{.25cm} \\
\lambda_3-\frac{1}{m(1-\beta)}\sum_{j=1}^m
\partial \left(.)^+\right|_{-d_j^Tw^{(j_\ell)}-\gamma^{(j_\ell)}}\ni 0,
\vspace{0.25cm}\\
\textbf{e}^Tx^{(j_\ell)}=1, \ \, 0\le T^{(j_\ell)}_1\perp (z^{(j_\ell)}-L)\ge 0,\ \, 0\le T^{(j_\ell)}_2\perp (U-z^{(j_\ell)})\ge 0, \ \,  0\le T^{(j_\ell)}_3\perp (w^{(j_\ell)}-L)\ge 0, \vspace{0.25cm}\\
\mbox{and} \ \, 0\le T^{(j_\ell)}_4\perp (U-w^{(j_\ell)})\ge 0.
\end{array}\right.
\end{equation}
Since for each $j_\ell$ we have 
$y^{(j_\ell)}_i\ge 0; \forall i\in \Gamma^+$ and $y^{(j_\ell)}_i\le 0; \forall i\in \Gamma^-$, from part (i), we can conclude  $x^{(j_\ell)}_i\ge 0; \forall i\in \Gamma^+$ and $x^{(j_\ell)}_i\le 0; \forall i\in \Gamma^-$.
Further, $\textbf{e}^Tx^{(j_\ell)}=1$ and $x^{(j_\ell)}-L\ge 0$ with $U-x^{(j_\ell)}\ge 0$.
Next, by injecting the second, third, and fourth equations of (\ref{sequence_kkt}) into its first one, we obtain the following:
\begin{eqnarray}\label{eqn: needed_1} 
2\lambda_1Ax^{(j_\ell)}-\lambda_2\Psi +t^{(j_\ell)}\textbf{e}
 &+ & 
\begin{bmatrix}
 2\rho^{(j_\ell)}(x^{(j_\ell)}-y^{(j_\ell)})_{\Lcal} \\ 0 \end{bmatrix}
 + 
\begin{bmatrix}  0 \notag \\ 2\rho^{(j_\ell)}(x^{(j_\ell)})_{\Lcal^c} \end{bmatrix}
\notag \\ &+& 
\lambda_2\delta \partial
\left( \|\cdot\|_1)\right|_{z^{(j_\ell)}-\phi}-T^{(j_\ell)}_1+T^{(j_\ell)}_2
\notag \\ &-& \frac{ \lambda_3}{m(1-\beta)}
\sum_{j=1}^m d_j\partial \left(\cdot)^+\right|_{-d_j^Tw^{(j_\ell)}-\gamma}
-T^{(j_\ell)}_3+T^{(j_\ell)}_4\notag \\ &\ni & 0.
\end{eqnarray}
From the other side, and using the second equation in (\ref{sequence_kkt}), after dropping the superscript ${(j_\ell)}$, we get
\begin{eqnarray*} 
\left(x-y\right)_\Lcal & = &
\left([x_{\Gamma^+};0]+[x_{\Gamma^-};0]+[x_{\Gamma^0};0]-y\right)_\Lcal
\notag
  \\ 
 & = &  
 \left([x_{\Gamma^+};0]+[x_{\Gamma^-};0]+[x_{\Gamma^0};0]-
 [x^+_{\Gamma^+};0]+[x^-_{\Gamma^-};0]-[x_{\Gamma^0};0]
 \right)_\Lcal
 \notag
  \\ 
 &= &  
 \left([x_{\Gamma^+};0]-
 [x^+_{\Gamma^+};0]+[x_{\Gamma^-};0]+[x^-_{\Gamma^-};0]
 \right)_\Lcal
 \notag
  \\
  &= &  
\left(-[x^-_{\Gamma^+};0]+
 [x^+_{\Gamma^-};0]
 \right)_\Lcal
\end{eqnarray*}
Therefore, (\ref{eqn: needed_1}) reduces to:
\begin{eqnarray}\label{eqn: needed_2} 
2\lambda_1Ax^{(j_\ell)}-\lambda_2\Psi +t^{(j_\ell)}\textbf{e}
 &- & 
 \underbrace{
\begin{bmatrix}
  2\rho^{(j_\ell)}{[x^-_{\Gamma^+};0]}^{(j_\ell)}_{\Lcal}
 \\ 0 \end{bmatrix}}_{:=T_5^{(j_\ell)}}
 + 
 \underbrace{
\begin{bmatrix}
  2\rho^{(j_\ell)}{[x^+_{\Gamma^-};0]}^{(j_\ell)}_{\Lcal}
 \\ 0 \end{bmatrix}}_{:=T_6^{(j_\ell)}}
 + 
\underbrace{\begin{bmatrix}  0 \notag \\ 2\rho^{(j_\ell)}(x^{(j_\ell)})_{\Lcal^c} \end{bmatrix}}_{:=\Omega^{(j_\ell)}}
\notag \\ &+& 
\lambda_2\delta \partial
\left( \|\cdot\|_1)\right|_{z^{(j_\ell)}-\phi}-T^{(j_\ell)}_1+T^{(j_\ell)}_2
\notag \\ &-& \frac{ \lambda_3}{m(1-\beta)}
\sum_{j=1}^m d_j\partial \left(\cdot)^+\right|_{-d_j^Tw^{(j_\ell)}-\gamma^{(j_\ell)}}
-T^{(j_\ell)}_3+T^{(j_\ell)}_4\notag \\ & \ni & 0.
\end{eqnarray}
Further, note that $0\le T_5^{(j_\ell)}\perp [x^+_{\Gamma^+};0]^{(j_\ell)}\ge 0$ and $0\le T_6^{(j_\ell)}\perp [x^+_{\Gamma^-};0]^{(j_\ell)}\ge 0$. 

We next show that 
$\{(t^{(j_\ell)},T_1^{(j_\ell)}+T_3^{(j_\ell)},T_2^{(j_\ell)}+T_4^{(j_\ell)},T_5^{(j_\ell)},T_6^{(j_\ell)}, \Omega^{(j_\ell)}\}$
is bounded under Robinson's condition on $(x^*,\gamma^*)$. 
Suppose not, and consider the following normalized sequence:
\[
(\wt t^{(j_\ell)}, \wt T_1^{(j_\ell)}+\wt T_3^{(j_\ell)},\wt T_2^{(j_\ell)}+\wt T_4^{(j_\ell)},\wt T_5^{(j_\ell)},\wt T_6^{(j_\ell)}, \wt \Omega^{(j_\ell)})
:= 
\frac{ (t^{(j_\ell)}, T_1^{(j_\ell)}+T_3^{(j_\ell)},T_2^{(j_\ell)}+T_4^{(j_\ell)},T_5^{(j_\ell)},T_6^{(j_\ell)}, \Omega^{(j_\ell)})} 
{ \| (t^{(j_\ell)}, T_1^{(j_\ell)}+T_3^{(j_\ell)},T_2^{(j_\ell)}+T_4^{(j_\ell)},T_5^{(j_\ell)},T_6^{(j_\ell)}, \Omega^{(j_\ell)})\|_2
}.
\]
Since this normalized sequence is bounded, it has a convergent subsequence, without loss of generality, itself, $(\wt t^{(j_\ell)}, \wt T_1^{(j_\ell)}+\wt T_3^{(j_\ell)},\wt T_2^{(j_\ell)}+\wt T_4^{(j_\ell)},\wt T_5^{(j_\ell)},\wt T_6^{(j_\ell)}, \wt \Omega^{(j_\ell)})$, whose limit is given by $(\wt t^*, \wt T_1^*+\wt T_3^*,\wt T_2^*+\wt T_4^*,\wt T_5^*,\wt T_6^*, \wt \Omega^*)$ such that $\| ({\color{black} \wt t^*},\wt T_1^*+\wt T_3^*,\wt T_2^*+\wt T_4^*,\wt T_5^*,\wt T_6^*, \wt \Omega^*)\|_2=1$.
After dividing both sides of (\ref{eqn: needed_2}) by $\| (t^{(j_\ell)}, T_1^{(j_\ell)}+T_3^{(j_\ell)},T_2^{(j_\ell)}+T_4^{(j_\ell)},T_5^{(j_\ell)},T_6^{(j_\ell)}, \Omega^{(j_\ell)})\|_2$ 
and then passing the limit $j_\ell\to \infty$, in virtue of part (i), the boundedness of the remaining terms in (\ref{eqn: needed_2}), the continuity of $\|.\|_1$ and $\max(x,0)$, and the boundedness of their subgradients, we obtain:
\begin{equation} \label{eqn:limit_condition}
   \wt t^*\textbf{e}-\wt T_5^*+\wt T_6^*+ \wt \Omega^*- \wt T_1^* + \wt T_2^*  -\wt T_3^* + \wt T_4^* \ni 0,
\end{equation}
where $\wt T^*_i\ge 0$ with $(\wt T^*_i)_{\mathcal{L}^c}=0;$ for $ i=1,2$ and  $\wt T_i^* \ge 0,$ for $i=1,2,3,4$ and $\wt \Omega^*_\Lcal=0$. In light of Robinson's condition (\ref{eqn: robinson}), there exists $(d,v,\bar v, \hat v, \tilde v)$ such that $\textbf{e}^Td=-\wt t^*, -d-\hat v=-\wt T_5^*, d-\tilde v = -T_6^*, -d-v=-\wt T^*_1-\wt T^*_3, d-\bar v=-\wt T^*_2-\wt T^*_4,$ and $d_{\mathcal{L}^c}=-\wt \Omega^*$. Since $d_{\Lcal^c} = - \wt \Omega^*_{\Lcal^c}$ and $\wt \Omega^*_\Lcal=0$, we see that $d^T \wt \Omega^* = -\| \wt \Omega^* \|^2_2$. By multiplying (\ref{eqn:limit_condition}) and using the equations above, we get the following:
$$
-(\wt t^*)^2-(\wt T^*_5-\hat v)^T \wt T^*_5+(\tilde v-\wt T^*_6)^T \wt T^*_6
-\|\wt \Omega^*\|_2^2
-(\wt T^*_1+\wt T^*_3-v)^T(\wt T^*_1+\wt T^*_3)+(\bar v-\wt T^*_2-\wt T^*_4)^T(\wt T^*_2+\wt T^*_4)=0,
$$
which leads to 
$$
\|((\wt t^*)^2,\wt T^*_5,\wt T^*_6,\wt T^*_1+\wt T^*_3, \wt T^*_2+\wt T^*_4,\wt \Omega^*)\|_2^2=
\underbrace{
\hat v^T\wt T_5^*}_{\le 0}+
\underbrace{
\tilde v^T\wt T^*_6}_{\le 0}+
\underbrace{v^T(\wt T^*_1+\wt T^*_3)}_{\le 0}+
\underbrace{\bar v^T(\wt T^*_2+\wt T^*_4)}_{\le 0};
$$
implying that $\|((\wt t^*)^2,\wt T^*_5,\wt T^*_6,\wt T^*_1+\wt T^*_3, \wt T^*_2+\wt T^*_4,\wt \Omega^*)\|_2^2=0$. This is a contradiction! Consequently, the sequence $\{(t^{(j_\ell)},T_1^{(j_\ell)}+T_3^{(j_\ell)},T_2^{(j_\ell)}+T_4^{(j_\ell)},T_5^{(j_\ell)},T_6^{(j_\ell)}, \Omega^{(j_\ell)}\}$ is bounded and has a convergent subsequence. Let $\{(\bar t,\bar T_1, \bar T_2,\bar T_5,\bar T_6,  \bar \Omega)$ be the limit point of this sequence. By passing the limit in (\ref{eqn: needed_2}) and applying the results from part (i),  one can see that the KKT conditions (\ref{eqn: actual_KKT_conditions for p}) are satisfied.

So far, we have proved that all the KKT conditions in (\ref{eqn: actual_KKT_conditions for p}) hold at $(x^*, \gamma^*)$, except for its second equation. To show this, recall the fourth equation in (\ref{sequence_kkt}):
\begin{equation*}
    \lambda_3-\frac{1}{m(1-\beta)}\sum_{j=1}^m
\partial \left(.)^+\right|_{-d_j^Tw^{(j_\ell)}-\gamma^{(j_\ell)}}\ni 0.
\end{equation*}
By taking the limit as $j_\ell \to \infty$, the boundedness of $\gamma^{(j_\ell)}$ shown in (\ref{lem: BCD_boundedness}), the continuity of $\max(x,0)$, and considering part (i), we see that  
\begin{equation*}
    \lambda_3-\frac{1}{m(1-\beta)}\sum_{j=1}^m
\partial \left(.)^+\right|_{-d_j^Tx^*-\gamma^*}\ni 0.
\end{equation*}
Therefore, we showed that all the KKT conditions hold at $(x^*, \gamma^*)$. 

Under Theorem 3.34 of \cite{ruszczynski2011nonlinear} and Theorem 2.3 of \cite{lu2013sparse}, by virtue of the linearity of constraints in (\ref{pr: p-original}), except for the cardinality constraint, we conclude that Algorithm \ref{algo: PD_method_p} converges to a Lu--Zhang minimizer of this problem under the given assumptions. 
\end{proof}

\section{Numerical Results} \label{sec: numerical}

In this section, we compare the performance of Algorithm \ref{algo: PD_method_p} with the algorithm in \cite{hamdi2024penalty}. The direct approach of solving (\ref{pr: p-original}) using CVX-Mosek is also reported for gap computation.
We use the data of the S\&P index for 2018-2021, with 120 stocks when $ \lambda_1=\lambda_2={1}/{3} $,  $ \phi=0 $,  $ U_{i}= -L_{i} = 0.2; \forall i\in [n], \delta= 0.002$, and $\beta=0.95$. Moreover, we apply the Geometric Brownian Motion (GBM) model in \cite{augustini2018} to generate scenarios.
In the PADM algorithm of \cite{hamdi2024penalty}, the initial penalty parameter is set to $1.2$, and it is updated by a factor of 3.
The inner loop is stopped when  $ ||(x^{s,l} , w^{s,l}) -(x^{s,l-1} , w^{s,l-1})||_ 1 \leq 10^{-5} $. It terminates with a partial minimum if $ ||x-  w||_1 \leq 10^{-5} $.
{\color{black} In Algorithm \ref{algo: PD_method_p}, the initial penalty parameter $\rho$ is set to $1.2$, and it is updated with the factor $r=3$. The inner loop is stopped for $\epsilon_I = 10^{-5}$, and the outer loop is terminated when the convergence criterion $\epsilon_O = 10^{-5}$ is met.}
All computations are performed in MATLAB R2017a on a 2.50 GHz laptop with 4 GB of RAM, and CVX   2.2 (\cite{grant2014cvx})  is used to solve the optimization models.
The results comparing returns,  risks,  Sharpe ratios, CVaR values, CPU times, and gaps for different $m$ values are reported in Table \ref{table1}. The gap in this table is $ {|f-\overline{f}|}/{(|f|+1)}$, where $f$ is the return (risk, Sharpe ratio, CVaR) for the direct solution approach, and $ \overline{f} $ is the return (risk, Sharpe ratio, CVaR) for the PADM algorithm or Algorithm \ref{algo: PD_method_p}.

These results show that both Algorithm \ref{algo: PD_method_p} and PADM are significantly faster than the direct solution approach. Further, Algorithm \ref{algo: PD_method_p} is about twice as fast as the PADM while having competitive gaps in returns, risks, CVaR, and Sharpe ratios for all $m$ values. These results for $m=3000$ are also depicted in  Figure \ref{KK1}  for different numbers of stocks, demonstrating the competitiveness of Algorithm \ref{algo: PD_method_p}. These results confirm that Algorithm \ref{algo: PD_method_p} is a better alternative to the direct solution approach than the PADM in \cite{hamdi2024penalty}.

\section{Conclusion} \label{sec: conclusion}
In conclusion, this paper designs a penalty decomposition algorithm customized for tackling the challenging sparse extended mean-variance-CVaR portfolio optimization problem in finance. This algorithm needs to solve a sequence of penalty subproblems. By employing a block coordinate method, we adeptly exploit every structure in the problem to manage each penalty subproblem, establishing their well-posed nature and deriving closed-form solutions wherever feasible. Our comprehensive convergence analysis demonstrates the efficacy of our introduced algorithm in efficiently reaching a Lu--Zhang minimizer of this non-differentiable and nonconvex optimization problem. Furthermore, extensive numerical experiments conducted on real-world datasets validate the practical applicability, effectiveness, and robustness of the algorithm we introduced across various evaluation criteria. Overall, this research contributes significantly to the field of portfolio optimization by offering a novel and efficient solution approach with promising practical implications in finance.

\begin{figure}[h]
	\centering
	\centering
	\subfloat[Return]{\includegraphics[width=7cm,height=5.2cm]{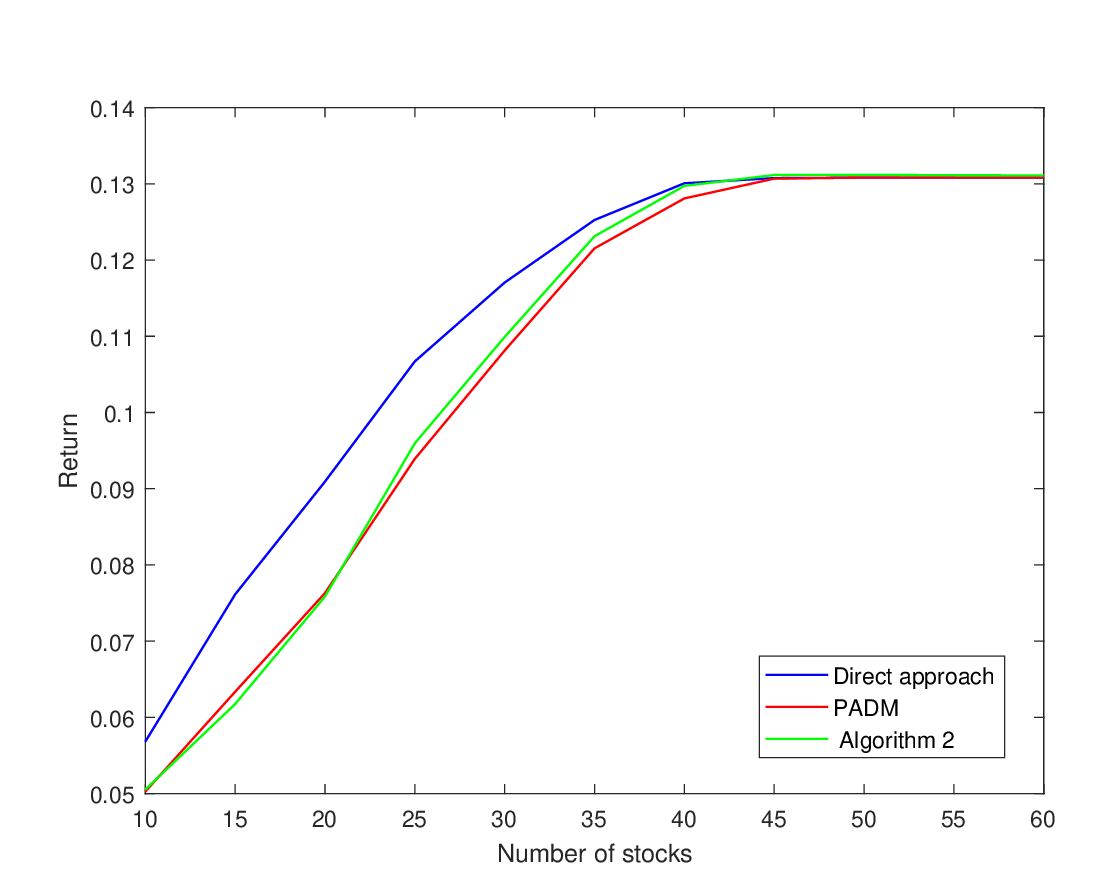}}
	\subfloat[Risk]{\includegraphics[width=7cm,height=5.2cm]{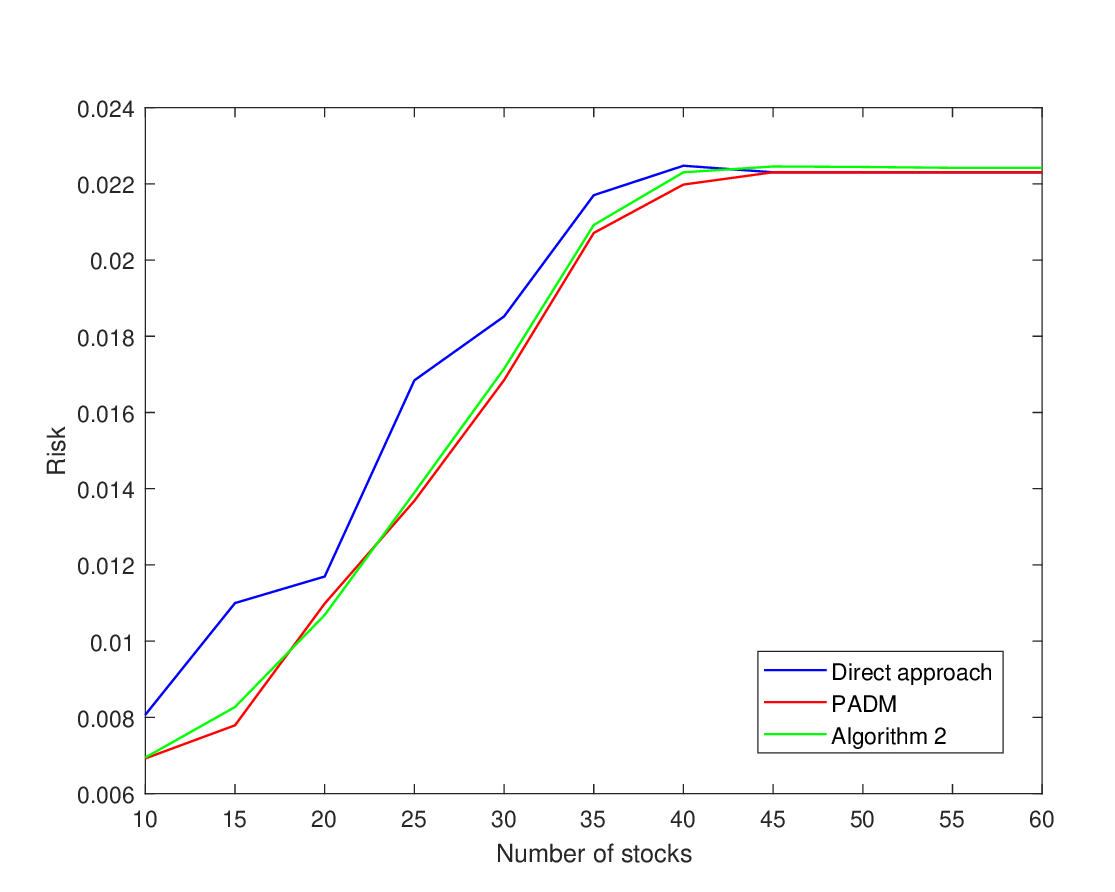}}\\
	\centering
	\subfloat[Sharpe ratio]{\includegraphics[width=7cm, height=5.2cm]{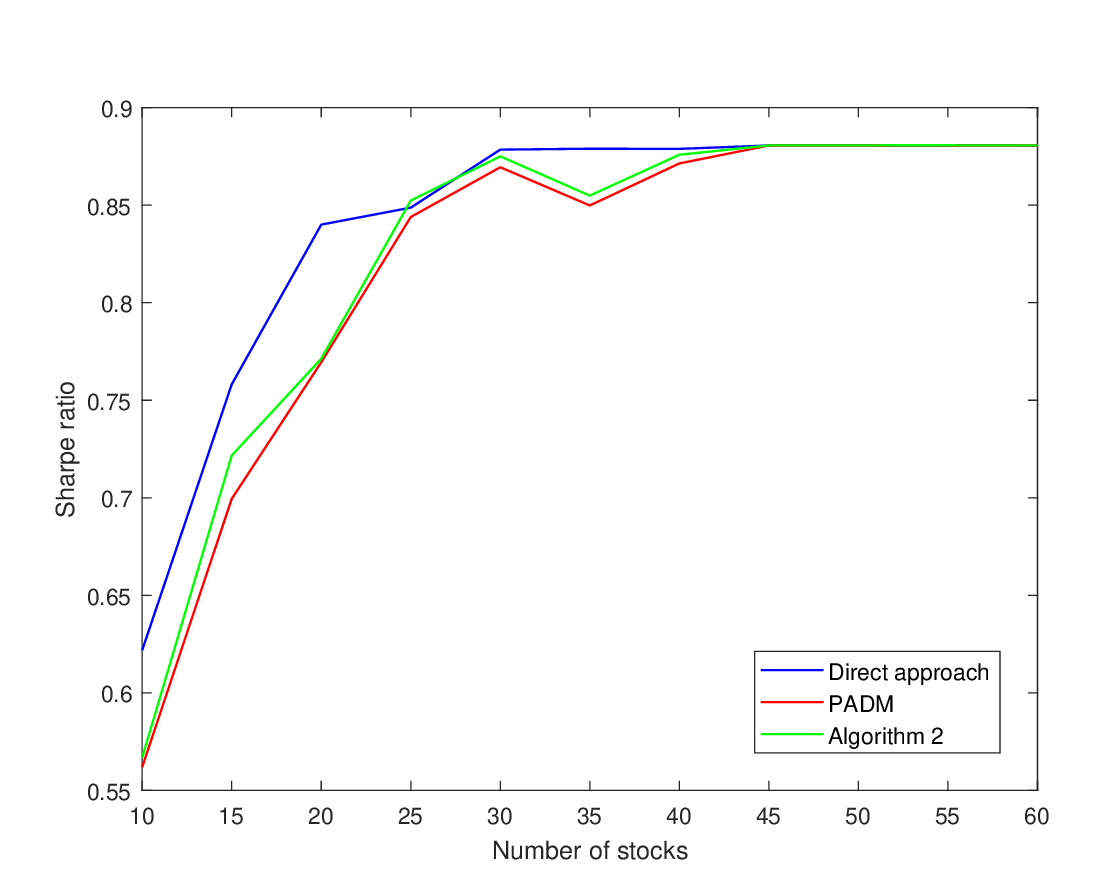}}
	\subfloat[CVaR]{\includegraphics[width=7cm,height=5.2cm]{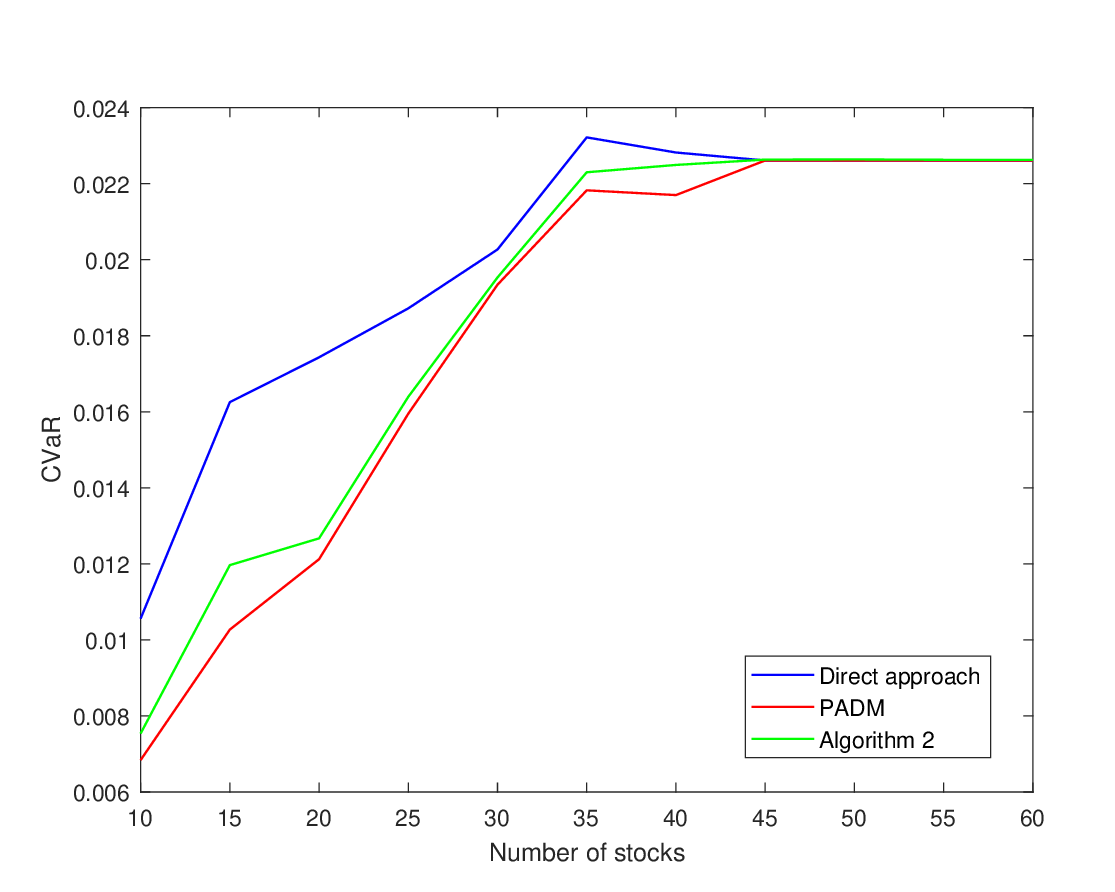}}\\
	\caption{Comparison of returns, risks,  Sharpe ratios, and CVaR for different $k$ values for 130 stocks of the S\&P index.
			\label{KK1}}
\end{figure}

\begin{center}
	\begin{table}[h]
    \centering
		\caption{ \small{Comparison of returns, risks, Sharpe ratios, CVaR and CPU times for the direct approach, PADM, and Algorithm \ref{algo: PD_method_p}  for different $ m $ values for S\&P index data with different confidence levels $ \beta $ and different number of scenarios $ m $ when $ \lambda_1= \lambda_2= \dfrac{1}{3}$, and $k=30$.}}
		\label{table1}
		\scalebox{0.85}{
			\begin{tabular}{lllllllllllllllllllll}
				\hline
				&&& $ \beta=0.95 $  \vspace{-0.5cm}
 &&&&\\
				  &  & &&&\\
				\cline{3-8}
				Model &   & m=1000 & m=3000 & m=5000 \\
				\hline
				\multirow{11}{*}{Direct approach}\\
				& Return  &0.1202 & 0.1195 & 0.1178  \\
				\\
				& Risk   & 0.0201& 0.0188 & 0.0161  \\
				\\
				& Sharpe ratio  & 0.8470 & 0.8728 & 0.0242   \\
				\\
			&  CVaR   & 0.0191 & 0.0211 &  0.0194    \\
		\\
		& CPU time   & 472.8531 & $ >  $2.0029e+03 & $  > $ 2.0029e+03   \\
				\\
				\hline
				\multirow{19}{*}{PADM}\\
	& Return  & 0.1134 & 0.1087 & 0.1109 \\
	\\
	& Risk&   0.0169 & 0.0159 & 0.0144  \\
	\\
	& Sharpe ratio   & 0.8709  & 0.8625 &  0.9242 \\
	\\
	&  CVaR &  0.0177 & 0.0183 & 0.0217  \\
	\\
	& CPU time  &  53.2698&  119.9078 & 312.8421\\
	\\
		& Return gap  & 0.0060 & 0.0097 & 0.0062 \\
	\\
		& Risk gap   & 0.0031 & 0.0028 & 0.0016 \\
	\\
		& Sharpe ratio gap & 0.0130 & 0.0055 &  0.0030  \\
	\\
		& CVaR gap  & 7.4739e-04 & 0.0019 & 0.0023 \\
	\\					\hline
				\multirow{19}{*}{Algorithm \ref{algo: PD_method_p}}\\
		& Return   & 0.1151 & 0.1108 & 0.1111 \\
			\\
			& Risk & 0.0173 & 0.0162 &  0.0147  \\
			\\
			& Sharpe ratio  & 0.8755 & 0.8695 &0.9174 \\
			\\
			&  CVaR   & 0.0188 & 0.0197 & 0.0224 \\
			\\
			& CPU time   & 38.4394 & 58.4394  & 61.2368 \\
			\\
			& Return gap   & 0.0045 & 0.0078 &  0.0060\\
			\\
			& Risk gap  & 0.0028 & 0.0025 & 0.0014\\
			\\
			& Sharpe ratio gap   &0.0154 & 0.0018 & 0.0065\\
			\\
			& CVaR gap  & 3.0090e-04 & 0.0015 & 0.0018 \\
			\\	
			\end{tabular}
		}
	\end{table}
\end{center}

\printbibliography

\section{Appendix} \label{sec: appendix}
\subsection{Proof of Lemma \ref{lem: solution_p_x}}
Technical proofs are reported below.
\begin{proof}
    Since $A\succeq 0$ and $\rho>0$, we have $A+3\rho I \succ 0$. Hence, this strictly convex problem has a unique solution, and for $t\in \mathbb R$, its KKT conditions are as follows:
\begin{equation} \label{eqn: KKT _Px}
2\lambda_1 Ax_* -\lambda_2\Psi+ 2\rho (3x_*-y-z-w) +t \textbf{e} =0, \quad \mbox{and} \quad \textbf{e}^Tx_*=1.
\end{equation}
Thus,  $2(\lambda_1A+3\rho I)x_*=2\rho (y+z+w)+\lambda_2\Psi-t \textbf{e}$. 
In virtue of (\ref{eqns: B and vartheta}), we see  $x_*=0.5 B^{-1}(  \vartheta-t \textbf{e})$, which, together with $1=\textbf{e}^Tx_*=0.5\textbf{e}^TB^{-1}(\vartheta -t \textbf{e})$, implies
\begin{equation*}
t = \frac{-1+0.5\textbf{e}^TB^{-1} \vartheta}{0.5\textbf{e}^TB^{-1}\textbf{e}},
\end{equation*}
This leads to (\ref{eqn: x_*}).
\end{proof}
\subsection{Proof of Lemma \ref{lem: solution_p_y}}
\begin{proof}
For any $y$ with $\| y \|_0 \le k$, we can write $y=[y_\Ical; y_{\Ical^c}]$ such that $y_\Ical =0$ for some index set $\Ical \subseteq [n]$ with $|\Ical|=n-k$. Hence, for any index set $\Ical$ with $|\Ical| =n-k$, $(P_y)$ can be written as the following indexed problem:
$$
\min_{ y \in \mathbb R^n} (\|  x_\Ical-y_\Ical \|^2_2+ \|  x_{\Ical^c}- y_{\Ical^c} \|^2_2) \qquad \text{subject to} \qquad  \sgn(\eta_i)y_i\ge 0; \quad \forall i,  \quad \text{and}  \quad y_{\Ical}=0, 
$$
which is equivalent to solving:
$$
\min_{\theta \in \mathbb R^{|\Ical^c|}} \| \theta -x_{\Ical^c} \|^2_2 \qquad \, \ \text{subject to} \,\qquad \sgn(\eta_i)\theta_i\ge 0; \quad \forall i\in \Ical^c.
$$
For any $i\in \Gamma^0 \cap \Ical^c$, we do not have any constraints, which  shows that $(\theta_*)_i=x_i$ for all $i\in  \Gamma^0  \cap \Ical^c$.
Thus, by letting $\Gamma=\supp(\mu)=\Gamma^+\cup \Gamma^-$,  it suffices to focus only on the following:
$$
\min_{\nu\in \mathbb R^{|\Gamma\cap \Ical^c|}} \| \nu-x_{\Gamma \cap\Ical^c} \|^2_2 \qquad \, \ \text{subject to} \,\qquad \sgn(\eta_i)\nu_i\ge 0; \quad \forall i\in \Gamma\cap \Ical^c.
$$
The KKT conditions for the global minimizer $\nu_*$ of this strictly convex problem can be written as: 
$$ x_{\Gamma \cap \Ical^c}=\nu_* - \zeta  \circ (\sgn(\eta))_{\Gamma \cap \Ical^c}
\quad \text{and} \quad 
\nu_*\circ (\sgn(\eta))_{\Gamma \cap \Ical^c}\ge 0, 
\quad 
\text{with}
\quad
\zeta  \ge 0
\quad 
\text{and}
\quad 
\zeta ^T(\sgn(\eta)\circ \nu_*)=0.
$$
Let $a=(\sgn(\eta)\circ x)_{\Gamma \cap\Ical^c}$. then it is known  that $a=a^+-a^-$ for  $a^+=\max(a,0)$, $a^-=\max(-a,0)$, and $a^{+}\perp a^-$. It is easy to see that $(\nu_*)_i={a_i^+}/{\sgn(\eta_i)}$ and $\zeta _i = {a^-_i}$ for any $i\in \Gamma\cap \Ical^c$, that is, 
$$
\nu_*=\sgn(\eta)\circ (\sgn(\eta)\circ x)^+_{\Gamma \cap\Ical^c}
\qquad
\text{and} 
\qquad
\zeta  = (\sgn(\eta)\circ x)^-_{\Gamma \cap\Ical^c}
$$
satisfy the given KKT conditions above. By recalling $\Gamma^+$ and $\Gamma^-$ defined  inside (\ref{pr: p-original}), we see
$$
(\nu_*)_{\Gamma^+\cap \Ical^c} =  x^+_{\Gamma^+\cap \Ical^c} \qquad  \text{and}  \qquad
(\nu_*)_{\Gamma^-\cap \Ical^c} =  -x^-_{\Gamma^-\cap \Ical^c}. 
$$
Therefore, we have $$
(\theta_*)_{\Gamma^0  \cap\Ical^c}=x_{\Gamma^0  \cap\Ical^c}, \qquad (\theta_*)_{\Gamma^+ \cap\Ical^c}=x^+_{\Gamma^+ \cap\Ical^c}, \qquad \text{and}   \qquad (\theta_*)_{\Gamma^- \cap\Ical^c}=-x^-_{\Gamma^- \cap\Ical^c}.
$$
So, for any index $\Ical$ specified above, the optimal value of the indexed problem  is given by 
\begin{eqnarray*} 
\|x_{\Ical^c}-\theta_*\|^2_2+\|x_\Ical\|^2_2
& = &
\|x_{\Gamma^+\cap\Ical^c}-x^+_{\Gamma^+\cap \Ical^c}\|^2_2+\|x_{\Gamma^+\cap\Ical}\|^2_2 
\notag
  \\ 
 & + &  
\|x_{\Gamma^-\cap\Ical^c}+x^-_{\Gamma^-\cap \Ical^c}\|^2_2+\|x_{\Gamma^-\cap\Ical}\|^2_2 
\notag
  \\ 
 &+&  
\|x_{\Gamma^0 \cap\Ical^c}-x_{\Gamma^0 \cap \Ical^c}\|^2_2+\|x_{\Gamma^0 \cap\Ical}\|^2_2,
\notag
\end{eqnarray*}
where
$$
 \|x_{\Gamma^+ \cap\Ical^c} - x^+_{\Gamma^+ \cap\Ical^c}  \|^2_2 + \|x_{\Gamma^+\cap\Ical}\|_2^2 =
  \|x_{\Gamma^+ \cap\Ical^c}\|_2^2 - \|x^+_{\Gamma^+ \cap\Ical^c}  \|^2_2 + \|x_{\Gamma^+\cap\Ical}\|_2^2 
  =
 \|x_{\Gamma^+}\|^2_2-\|x^+_{\Gamma^+\cap\Ical^c}\|_2^2,
$$
$$
 \|x_{\Gamma^- \cap\Ical^c} + x^-_{\Gamma^- \cap\Ical^c}  \|^2_2 + \|x_{\Gamma^-\cap\Ical}\|_2^2 = 
  \|x_{\Gamma^- \cap\Ical^c}\|_2^2-\| x^-_{\Gamma^- \cap\Ical^c}  \|^2_2 + \|x_{\Gamma^-\cap\Ical}\|_2^2 = 
 \|x_{\Gamma^-}\|^2_2-\|x^-_{\Gamma^-\cap\Ical^c}\|_2^2,
$$
and
$$
 \|x_{\Gamma^0  \cap\Ical^c} -x_{\Gamma^0  \cap\Ical^c}  \|^2_2 + \|x_{\Gamma^0 \cap\Ical}\|_2^2 = \|x_{\Gamma^0}\|^2_2-\|x_{\Gamma^0 \cap\Ical^c}\|_2^2.
$$
Consequently, the optimal value becomes 
\begin{eqnarray*} 
\|x_{\Ical^c}-\theta_*\|^2_2+\|x_\Ical\|^2_2
& = &
 \|x_{\Gamma^+}\|^2_2-\|x^+_{\Gamma^+\cap\Ical^c}\|_2^2
\notag
  \\ 
 & + &  
\|x_{\Gamma^-}\|^2_2-\|x^-_{\Gamma^-\cap\Ical^c}\|_2^2
\notag
  \\ 
 &+&  
 \|x_{\Gamma^0}\|^2_2-\|x_{\Gamma^0 \cap\Ical^c}\|_2^2
 \notag 
 \\
 &=&
 \|x\|^2_2
 -\|x^+_{\Gamma^+\cap\Ical^c}\|_2^2
-\|x^-_{\Gamma^-\cap\Ical^c}\|_2^2
-\|x_{\Gamma^0 \cap\Ical^c}\|_2^2.
\notag
\end{eqnarray*}
This implies that the minimal value of $(\ref{pr: py})$ is achieved when 
$
\|x^+_{\Gamma^+\cap\Ical^c}\|_2^2
+\|x^-_{\Gamma^-\cap\Ical^c}\|_2^2
+\|x_{\Gamma^0 \cap\Ical^c}\|_2^2
$ 
is maximal, or equivalently when $\Ical^c = \Jcal(
[x^+_{\Gamma^+};0]-[x^-_{\Gamma^+};0]+[x_{\Gamma^0};0]
,k)$, where $\Jcal(.,k)$ selects the $k$ largest components in absolute value. Therefore, a minimizer $y^*$ is defined in (\ref{eqn: y_*}).
\end{proof}

\subsection{Proof of Lemma \ref{lem: solution_p_z}}
\begin{proof}

First, note that if $\lambda_2\delta=0$, then we have:
\begin{equation*}
\begin{aligned}
\min_{z\in \mathbb{R}^n} \quad & 
\|x-z\|_2^2
 \qquad \textrm{subject to} \qquad
 L\le z\le U.
\end{aligned}
\end{equation*}
In this case ${z_*}_i=U_i$ if $x_i>U_i$, ${z_*}_i=L_i$ if $x_i<L_i$, and otherwise ${z_*}_i=x_i$, for any $i\in [n]$, which can be rewritten as 
\begin{equation} 
    z_*= \max(L,\min(x,U)).
\end{equation}
Next, suppose that $\lambda_2\delta\ne 0$, due to the separability property of (\ref{pr: pz}), letting $\tilde {z_i}=z_i-\phi_i, \tilde x_i=x_i-\phi_i, \tilde L_i=L_i-\phi_i,$  and $\tilde  U_i= U_i-\phi_i$, for $\kappa = {(\lambda_2\delta)}/{(2\rho)}$, we shall focus on the following:
\begin{equation*}
\min_{ \tilde z_i\in \mathbb{R}} \quad 
\frac{1}{2}( \tilde z_i-\tilde x_i)^2+\kappa| \tilde z_i |
 \qquad \textrm{subject to} \qquad
 \tilde L_i\le  \tilde z_i\le \tilde U_i.
\end{equation*}
Let us assume that $\tilde z_i> 0$; then we have $\tilde z_i=\tilde x_i -\kappa$. Thus, if $ \tilde x_i -\kappa>0$ and $\tilde L_i\le \tilde x_i-\kappa \le \tilde U_i$, then $ \tilde z_{*i}=x_i -\kappa$. If $0<\tilde x_i -\kappa<\tilde L_i$, we must have $\tilde z_{*i}=\tilde L_i$. If $\tilde x_i-\kappa>\max(\tilde U_i, 0)$, then $\tilde z_{*i}=\tilde U_i$.
In the case of  $\tilde z_i< 0$, taking the derivative gives $\tilde z_i= \tilde x_i+\kappa$. Thus, if  $\tilde x_i+\kappa < 0$, and $\tilde L_i \le x_i+\kappa \le \tilde U_i$, we have $z_{*i}=x_i+\kappa$. If  $\tilde x_i+\kappa < 0$, and $\tilde x_i+\kappa < L_i$, then we must have $\tilde z_{*i}=L_i$. Otherwise, $\tilde x_i+\kappa <0,$ and $\tilde x_i+\kappa > \tilde U_i$, so we have $\tilde z_{*i}=\tilde U_i$. For $\tilde z_i=0$, we also consider three different possibilities. In this case, the optimality conditions are: $-\tilde x_i+\kappa \partial (0) -T_1+T_2 \ni 0, \, 0\le T_1\perp (\tilde L_i-\tilde z_i)\le 0, \, \text{and}\, 0\le T_2\perp(\tilde z_i-\tilde U_i)\le 0$.
It can be shown that we must have $|\tilde x_i|\le \kappa$, and depending on three scenarios: that $0<\tilde L_i, 0\in [\tilde L_i, \tilde U_i],$ or $0> \tilde U_i$, we obtain the solutions reported below:

\begin{equation*}
\tilde z_{*i}:= \left\{ 
\begin{array}{ll}
\tilde x_i-\kappa  & \qquad \mbox{if \quad  $\tilde L_i\le \tilde  x_i-\kappa\le \tilde U_i$} \qquad \text{and} \qquad
\mbox{$\tilde  x_i-\kappa>0$};
\\
\tilde x_i+\kappa & \qquad \mbox{if \quad  $\tilde L_i\le \tilde x_i+\kappa \le \tilde U_i$}\qquad \text{and} \qquad
\mbox{$\tilde  x_i+\kappa<0$};
\\
L_i & \qquad \mbox{if \quad $0<\tilde x_i-\kappa < \tilde L_i$} \qquad \text{or} \qquad \mbox{$\tilde x_i+\kappa < \min(0, \tilde L_i)$};
\\
U_i &  \qquad \mbox{if \quad  $\tilde  x_i-\kappa>\max(\tilde U_i,0)$} \qquad \text{or} \qquad
\mbox{$\tilde U_i<\tilde  x_i+\kappa<0$};
\\
0 &  \qquad \mbox{if \quad $-\kappa\le \tilde  x_i\le \kappa$} \qquad \text{and} \qquad
\mbox{$\tilde L_i\le 0\le \tilde U_i$};
\\
\tilde L_i &  \qquad \mbox{if \quad $-\kappa\le \tilde  x_i\le \kappa$} \qquad \text{and} \qquad
\mbox{$0< \tilde L_i$};
\\
\tilde U_i &  \qquad \mbox{if \quad $-\kappa\le \tilde  x_i\le \kappa$} \qquad \text{and} \qquad
\mbox{$0> \tilde U_i$};
\end{array} \right. 
\end{equation*}
If we summarize the above based on the original variables, we obtain (\ref{lem: solution_p_z}).
\end{proof}

\subsection{Proof of Lemma \ref{lem: boundedness of p_w}}
\begin{proof}
We first show the existence of a solution. It is enough to prove that the objective function is bounded below over its feasible set. Since $\beta \in (0,1)$, we have $1-{1 \over 1-\beta} < 0$ such that 
\begin{eqnarray}\label{ineq_w_gamma}
\frac{c(w,\gamma) -\rho \|x-w\|_2^2}{\lambda_3}&=& \gamma + {1 \over m(1-\beta)} \sum_{j=1}^m  \max(0, -w^T d_j-\gamma) \notag\\
&=& \gamma + \sum_{j=1}^m  \max(0, -{1 \over m(1-\beta)} w^T d_j-{1 \over m(1-\beta)} \gamma) \notag\\
&=& \sum_{j=1}^m  \max({\gamma \over m}, -{1 \over m(1-\beta)} w^T d_j + {\gamma \over m}(1-{1 \over 1-\beta}) ) \notag\\
&\ge& \sum_{j=1}^m  {1 \over   { \beta \over 1-\beta}} (-{1 \over m(1-\beta)} d_j^T w ) \\
&=& -{ 1\over m\beta} (\sum_{j=1}^m  d_j)^T w\notag \\
&\ge & -\frac{1}{m\beta}\|\sum_{j=1}^m  d_j\|\|w\| \notag \\
&\ge&  -\frac{1}{m\beta}\|U\|\|\sum_{j=1}^m  d_j\| , \notag
\end{eqnarray}
where we used that $\max(t, u-a t) \ge {u \over 1 + a}$ whenever $a>0$ for $t=\gamma/m$ and $a=\beta/(1-\beta)$, and $u=-d_j^Tw/m(1-\beta)$ in (\ref{ineq_w_gamma}). This implies that
\begin{eqnarray}
c(w,\gamma) &\ge& -\frac{\lambda_3}{m\beta}\|U\|\|\sum_{j=1}^m  d_j\|+\rho \|x-w\|_2^2 
\notag\\
&\ge &  -\frac{\lambda_3}{m\beta}\|U\|\|\sum_{j=1}^m  d_j\| ,
\end{eqnarray}
which establishes the problem (\ref{pr: pw}) is bounded below and thus it has a solution.

Note that $w_*$ is unique because the objective function is strictly convex with respect to $w$. 
Once $w_*$ is known, it can be shown that $\gamma_*$ is either zero or must be equal to $-w_*d_{j_*}$ for some $j_*\in [m]$. To see this, without loss of generality, suppose for $j=1,\dots, \bar m$ with $\bar m\le m$, we have $-w_*^Td_j-\gamma>0$; thus, (\ref{pr: pw}) reduces to 
\begin{equation*} 
\begin{aligned}
\min_{\gamma\in  \mathbb R}\quad &
\gamma+\frac{1}{m(1-\beta)}
\sum_{j=1}^{\bar m} (-w_*^Td_j-\gamma)= \min_{\gamma\in  \mathbb R}\quad & \frac{m(1-\beta)-\bar m}{m(1-\beta)}\gamma.
\end{aligned}
\end{equation*}
If $m(1-\beta)-\bar m\ge 0$, then $\gamma_*=0$; otherwise, we have $\gamma_*= \min_{j\in [\bar m]}-w_*^Td_j$.
\end{proof}

\subsection{Proof of Lemma \ref{lem: BCD_boundedness}}

\begin{proof}
For simplicity, we drop the subscript $l$ whenever it is clear.
First, simply $\|w\|\le \|\max(-L,U)\|$. In Lemma \ref{lem: boundedness of p_w},  we also proved that $\gamma=-w^Td_{j}$ for some $j\in [m]$. Thus, we have $|\gamma|\le \|w\|\max_{j\in [m]} \|d_j\| \le \|\max(-L,U)\|\max_{j\in [m]} \|d_j\|$, which implies that
\begin{equation}
    \max\{|\gamma|, \|w\|\}\le \|\max(-L,U)\|\max(1, \max_{j\in [m]}\|d_j\|):=\bar C
\end{equation}
Also, clearly $\|z\|\le \|\max(-L,U)\|$. Since $\|y\|\le \|x\|$ from (\ref{def: gen_sparse}), we remain to prove that $\|x\|$ is bounded.
Recall the definitions in (\ref{eqns: B and vartheta}) and let $\alpha = 0.5\textbf{e}^TB\textbf{e}$. We see 
$$\|B^{-1}\|=\frac{1}{{\lambda_{\min}(B)}}=\frac{1}{{\lambda_1\lambda_{\min}(A)+3\rho}}\le \frac{1}{3{\rho}},$$
and 
$$2\alpha =\textbf{e}^TB^{-1}\textbf{e}\ge \|\textbf{e}\|^2\lambda_{\min}(B^{-1})
=\|\textbf{e}\|^2\lambda^{-1}_{\max}(B)=
\frac{n}{\lambda_{\max}(\lambda_1A+3\rho)}=
\frac{n}{\lambda_1\lambda_{\max}(A)+3\rho}.$$
Also, 
$$\|\vartheta\|\le
3 \rho(\|y\|+\|z\|+\|w\|)+\lambda_2(\|\mu\|+r_c\|h\|)
\le
3 \rho(\|x\|+2\|\max(-L,U)\|)+\lambda_2(\|\mu\|+r_c\|h\|)\le 3\rho \|x\|+\mathcal{O}(1)
$$
Thus, 
the equation (\ref{eqn: x_*}) together with $\|y\|\le \|x\|$ leads to
\begin{equation*}
\begin{aligned}
\|x\| 
&\le
\frac{1}{2} 
\|B^{-1}\|
\left(
\|\vartheta\|
+
\frac{1+0.5|\textbf{e}^TB^{-1}\vartheta|}{\alpha}\|\textbf{e}\|
\right)\\ 
&
\le
\frac{1}{2} 
\|B^{-1}\|
\left(
\|\vartheta\|
+
\frac{1+0.5 \|\textbf{e}\|\|B^{-1}\|\|\vartheta\|}{\alpha}\|\textbf{e}\|
\right)\\ 
& 
\le
\frac{1}{2} 
\|B^{-1}\|
\left(
\|\vartheta\|
+
\frac{1+0.5 \|\textbf{e}\|\|\vartheta\|}{6n\rho}(
\lambda_1 \lambda_{\max}(A)
+3\rho)\|\textbf{e}\|
\right)\\ 
&\le 
\frac{1}{36\rho^2}((7.5\rho+0.5\lambda_1\lambda_{\max}(A)) \|\vartheta\|+(\lambda_1 \lambda_{\max}(A)+3\rho)\|\textbf{e}\|)
\\
&\le 
\frac{1}{36\rho^2}(22.5\rho^2\|x\|+\mathcal{O}(\rho)).
\end{aligned}
\end{equation*} 
Thus, whenever $\rho\ge 1$, we have 
$
\|x\| 
\le {\mathcal{O}(1)}/{13.5\rho}\le \mathcal{O}(1).$ Consequently, we can say that $\max\{\|x_{l}\|, $ $\|y_l\|, \|z_{l}\|,\|w_l\|, |\gamma_l|\}$ is bounded above, and its bound is independent from $\rho$.
\end{proof}

\end{document}